\documentclass[11pt]{article}

\usepackage[hmargin=2cm,vmargin=1cm,includefoot,includehead]{geometry}

\usepackage{mathtools,amssymb,amsthm,bm,empheq}
\usepackage[all]{xy}
\usepackage{eqparbox}

\usepackage[backend=biber,style=alphabetic,sorting=nyt]{biblatex}
\renewbibmacro{in:}{}

\usepackage[colorlinks = true, pdfstartview = FitV, linkcolor = blue, citecolor = blue,
            urlcolor = blue]{hyperref}

\newcommand{\eqmath}[3][l]{%
  \eqmakebox[#2][#1]{$\displaystyle#3$}%
}

\newcommand{\pp}[2]{(\!( #1 , #2 )\!)}
\newcommand{\bb}[2]{\{\!\{ #1 , #2 \}\!\}}

\DeclareRobustCommand{\Chi}{{\mathpalette\irchi\relax}}
\newcommand*{\irchi}[2]{\raisebox{\depth}{$#1\chi$}}

\newcommand*{\andd}{\;\&\;}
\newcommand*{\rimplies}{\implies}
\newcommand*{\comma}{\:,\:}
\newcommand*{\define}[1]{\textbf{#1}}
\renewcommand*{\phi}{\varphi}

\newcommand*{\calL}{\mathcal{L}}
\newcommand*{\calP}{\mathcal{P}}
\newcommand*{\VR}{\mathrm{VR}}
\newcommand*{\vrMV}{\mathrm{vrMV}}
\newcommand*{\fmMV}{\mathrm{fmMV}}
\newcommand*{\MV}{\mathrm{MV}}
\newcommand*{\MF}{\mathrm{MF}}
\newcommand*{\FM}{\mathrm{FM}}

\DeclareMathOperator{\DV}{DV}
\DeclareMathOperator{\OC}{OC}
\DeclareMathOperator{\FV}{FV}

\newcommand*{\Schemes}{\mathrm{Schemes}}

\newcommand*{\Schemeshypfree}{{\Schemes_0}}
\newcommand*{\FMwithhyp}{{\FM_{\mathrm{with}\text{-}\mathrm{hyp}}}}

\DeclareMathOperator{\Inst}{Inst}
\DeclareMathOperator{\InstfmMVfree}{{\Inst_{\fmMV\text{-}\mathrm{free}}}}

\DeclareMathOperator{\val}{val}

\DeclareMathOperator{\Term}{Term}
\DeclareMathOperator{\Nat}{Nat}

\theoremstyle{plain}
\newtheorem{theorem}{Theorem}[section]
\newtheorem{proposition}[theorem]{Proposition}

\newtheorem{lemma}[theorem]{Lemma}

\theoremstyle{definition}

\theoremstyle{remark}
\newtheorem{remark}[theorem]{Remark}

\title{Independence questions in a finite\\axiom-schematization of first-order logic}
\author{Beno\^it Jubin}
\date{16 February 2022}

\begin{document}

\maketitle

\begin{center}
\textit{Dedicated to the memory of Norman Megill}
\end{center}

\begin{abstract}
We review some independence results in a finite axiom-schematization of classical first-order
logic introduced by Norman Megill.  We also prove that a certain axiom scheme of this system is
independent although all of its instances are provable from the other axiom schemes.
\end{abstract}

\setcounter{tocdepth}{2}
\tableofcontents

\section*{Introduction}
\addcontentsline{toc}{section}{Introduction}

Many mathematical theories which are not finitely axiomatizable can nevertheless be axiomatized
by a finite number of axiom schemes.
Examples are classical first-order logic and Zermelo--Fraenkel set theory.
If one wants to consider these theories as genuinely finitely axiomatized, one has not only to
exhibit a finite number of axiom schemes, but also to carry out the proofs at the scheme level
rather than at the object level, as is traditionally done.
This gives rise to interesting questions on the relationship between proofs at these two levels,
and in particular the relationship between ``scheme-independence'' and
``object-independence\rlap{.}''

We present such a finite axiom-schematization of classical first-order logic mainly due to
Megill~\cite{megill} building on earlier work by Tarski~\cite{tarski}, Kalish--Montague~\cite{km},
and Monk~\cite{monk}.
It has the advantage of requiring only very simple metalogic, in that it does not use the notions
of bound and free variables, but only the notions of a variable occurring in a formula and of two
variables being distinct.
It has no notion of proper substitution, but only plain substitution.

In this article, ``first-order logic'' means classical classical\footnote{
One ``classical'' is for ``classical propositional calculus'', as opposed to intuitionistic or
minimal propositional calculus, and the other ``classical'' is in opposition to free logic.}
one-sorted first-order logic with equality and no terms.
Before formally defining the required notions, in particular that of scheme, in
Section~\ref{sec:formal},\footnote{
Defined terms are written in boldface.}
we give the axiom schemes of this system, indicating simply for now that ``$\DV(x, \phi)$''
(resp.\ ``$\DV(x, y)$'') means that in the instances of the corresponding scheme, the variable
substituted for~$x$ should not occur in the formula substituted for~$\phi$ (resp.\ the variables
substituted for~$x$ and for~$y$ should be distinct).
As explained below, the symbols ``$\andd$'' and ``$\rimplies$'' are not connectives of the logic
and are to be read informally as ``and'' and ``implies''.
Also, ``\texttt{minimp}'' stands for ``minimal implicational calculus'', as explained below
together with the meaning of the other labels.
We choose a system with more numerous and weaker axiom schemes, which permits the study of
several subsystems axiomatizing various well-studied logics.

\begin{empheq}[left={\makebox[4em][r]{\texttt{propcalc}}\empheqlbrace\qquad}]{align}
&\eqmath{A}{\phi \andd \phi \to \psi \rimplies \psi} \tag{\texttt{mp}}\\
&\eqmath{A}{\phi \to ((\psi\to\chi) \to (((\theta\to\psi) \to (\chi\to\tau)) \to (\psi\to\tau)))}
 \tag{\texttt{minimp}}\\
&\eqmath{A}{((\phi\to\psi) \to \phi) \to \phi} \tag{\texttt{peirce}}\\
&\eqmath{A}{(\phi\to\neg\psi) \to (\psi\to\neg\phi)} \tag{\texttt{contrap}}\\
&\eqmath{A}{\neg\phi \to (\phi\to\psi)} \tag{\texttt{notelim}}
\end{empheq}
\vspace{-0.3em}
\begin{empheq}[left={\makebox[4em][r]{\texttt{modal bloc}}\empheqlbrace\qquad}]{align}
&\eqmath{A}{\phi \rimplies \forall x \phi} \tag{\texttt{gen}}\\
&\eqmath{A}{\forall x (\phi \to \psi) \to (\forall x \phi \to \forall x \psi)} \tag{\texttt{ALLdistr}}\\
&\eqmath{A}{\forall x \phi \to \phi} \tag{\texttt{spec}}\\
&\eqmath{A}{\neg \forall x \phi \to \forall x \neg \forall x \phi} \tag{\texttt{modal5}}
\end{empheq}
\vspace{-0.3em}
\begin{empheq}[left={\makebox[4em][r]{}}\phantom{\empheqlbrace}\qquad]{align}
&\eqmath{A}{\phi \to \forall x \phi \comma \DV(x, \phi)} \tag{\texttt{vacGen}}\\
&\eqmath{A}{\forall x \forall y \phi \to \forall y \forall x \phi} \tag{\texttt{ALLcomm}}
\end{empheq}
\vspace{-0.3em}
\begin{empheq}[left={\makebox[4em][r]{\texttt{equality}}\empheqlbrace\qquad}]{align}
&\eqmath{A}{x \equiv x} \tag{\texttt{EQrefl}}\\
&\eqmath{A}{x \equiv y \to y \equiv x} \tag{\texttt{EQsymm}}\\
&\eqmath{A}{x \equiv y \to (y \equiv z \to x \equiv z)} \tag{\texttt{EQtrans}}
\end{empheq}
\vspace{-0.3em}
\begin{empheq}[left={\makebox[4em][r]{}}\phantom{\empheqlbrace}\qquad]{align}
&\eqmath{A}{x \equiv x \to \neg \forall y \neg y \equiv x \comma \DV(x, y)} \;\;\tag{\texttt{denot}}\\
&\eqmath{A}{\forall x (x \equiv y \to (\phi \to \forall x (x \equiv y \to \phi))) \comma
 \DV(x, y) \comma \DV(y, \phi)} \tag{\texttt{subst}}\\
&\eqmath{A}{\forall x x \equiv y \to (\forall x \phi \to \forall y \phi)} \tag{\texttt{ALLeq}}\\
&\eqmath{A}{\neg \forall x x \equiv y \to (\neg \forall x x \equiv z \to
 (y \equiv z \to \forall x y \equiv z))} \tag{\texttt{genEq}}
\end{empheq}

We call this system the Tarski--Monk--Megill system\footnote{
The work of Kalish--Montague consisted in proving that the scheme of specialization \texttt{spec}
is object-provable in a related system, and in proving some independence results, but we think
that this contributed slightly less to the final form of the system~\texttt{TMM} than the works
of the three cited authors.}
and denote it by~\texttt{TMM}.
It has some variants, that we also informally call the \texttt{TMM} system.
We present some of these variants in Appendix~\ref{app:variants}, together with comments on some of
these axiom schemes.
Many variants have fewer axiom schemes, which can be an advantage for some applications (for
instance, finding models), but makes them less modular.
Here, we have chosen, on the contrary, to have more (conjecturally independent) axiom schemes,
allowing for a piecemeal presentation, in which several subsets of axiom schemes axiomatize
well-known logics (see Figure~\ref{fig:subsystems} of Appendix~\ref{app:variants}).

For first-order theories on a given language, one adds to the above schemes $n$ ``predicate axiom
schemes'' for each $n$-ary nonlogical predicate, sometimes called the ``equality axiom schemes''
associated with that predicate.
For example, if the language contains exactly one nonlogical predicate, denoted by~$\in$, which is
binary and written in infix notation, then one adds the two predicate axiom schemes
\begin{gather}
x \equiv y \to (x \in z \to y \in z) \tag{\texttt{ax-$\in_1$}},\\
x \equiv y \to (z \in x \to z \in y) \tag{\texttt{ax-$\in_2$}}.
\end{gather}
If $\calL$ is a language and emphasis is needed, then we denote by $\texttt{TMM}_\calL$ the full
system with the appropriate predicate axiom schemes.

The labels we used for these axiom schemes abbreviate, respectively:
modus ponens,
minimal implicational calculus,
Peirce's law,
contraposition,
``not'' elimination,
rule of generalization,
``forall distributes over implication\rlap{,}''
specialization,
axiom corresponding to the modal logic axiom \textbf{5},
vacuous generalization,
``forall quantifiers commute\rlap{,}''
equality is reflexive (resp.\ symmetric, transitive),
denotation\rlap{,}\footnote{
Any variable ``denotes'' in the sense of free logic.}
substitution, ``forall quantifiers over equal variables\rlap{,}'' generalized equality.
We make use of the standard translation between modal logic and monadic first-order logic
(necessity maps to ``$\forall x$'' and possibility to ``$\exists x$'', which we use as a
shorthand for ``$\neg \forall x \neg$'').
Other common names for some of these axiom schemes are
$\texttt{ALLdistr} = \texttt{modalK} = \texttt{kripke}$ and
$\texttt{spec} = \texttt{modalT}$.
Other schemes sometimes used as axioms are:
\begin{gather}
\forall x \neg \phi \to \neg \forall x \phi, \tag{\texttt{modalD}}\\
\neg \phi \to \forall x \neg \forall x \phi, \tag{\texttt{modalB}}\\
\forall x \phi \to \forall x \forall x \phi. \tag{\texttt{modal4}}
\end{gather}
The scheme \texttt{spec} implies \texttt{modalD} over \texttt{propcalc} (defined below).

We define the following subsystems of \texttt{TMM}:
\begin{align*}
\texttt{propcalc} &\coloneqq \{ \texttt{mp}, \texttt{minimp}, \texttt{peirce},
                                \texttt{contrap}, \texttt{notelim} \},\\
\texttt{EQ}       &\coloneqq \{ \texttt{EQrefl}, \texttt{EQsymm}, \texttt{EQtrans}\},\\
\texttt{T}        &\coloneqq \texttt{propcalc} \cup \texttt{EQ} \cup
                             \{ \texttt{gen}, \texttt{ALLdistr}, \texttt{vacGen}, \texttt{denot} \},\\
\texttt{TM}       &\coloneqq \texttt{TMM} \setminus \{ \texttt{ALLeq}, \texttt{genEq} \}.
\end{align*}

The notions of completeness and independence, recalled in Section~\ref{sec:formal}, are relative
to schemes and proofs at the scheme level;
the notions of object-completeness and object-independence are relative to the instances at the
object level of these schemes, and proofs between them.

The system \texttt{T} was proved to be object-complete by Tarski~\cite{tarski} and
object-independent by Kalish--Montague~\cite{km}.
The system \texttt{TMM} was proved to be complete by Megill~\cite{megill}.
The question of independence of its axiom schemes is still open, and the main new results in this
article are steps in that direction.

Namely, we build on Kalish--Montague's and Monk's proofs of independence (and in several cases provide
new proofs)\footnote{The proof of the independence of the rule of generalization is new, and for
the axioms of propositional calculus, one truth table is new and the others are classical.} to
prove the independence of the axiom schemes of~\texttt{T} and of~\texttt{subst} and~\texttt{modal5}
in~\texttt{TMM}.\footnote{
With some restrictions detailed in the text for the axiom schemes of the equality bloc.}
Using new notions of ``supertruth\rlap{,}'' we prove the independence
of~\texttt{ALLcomm} in $\texttt{TMM} \setminus \{\texttt{spec}, \texttt{ALLeq}\}$
(Proposition~\ref{prop:ALLcomm}),
of~\texttt{ALLeq} in $\texttt{TMM} \setminus \{\texttt{spec}\}$
(Proposition~\ref{prop:ALLeq}),
and of~\texttt{spec} in~\texttt{TMM}
(Proposition~\ref{prop:spec}).
The remaining open questions are the independence of the latter three axiom schemes in~\texttt{TMM},
as well as of~\texttt{genEq}.

\begin{remark}
A fundamental insight of Tarski, which makes these systems with weaker metalogic work, is that
if~$x$ and~$y$ are disjoint (that is, the variables substituted for~$x$ and for~$y$ should be
distinct), then the formula $\forall x (x \equiv y \to \phi)$, which we denote by $[y/x]\phi$,
is equivalent to the result of the proper substitution of~$y$ for~$x$ in~$\phi$.
With this notation, the scheme \texttt{subst} can be written
$[y/x](\phi \to [y/x]\phi), \DV(x, y), \DV(y, \phi)$.
\end{remark}

\begin{remark}\label{rmk:mpe}
The finitary property of the axiom-schematization \texttt{TMM} makes it well-suited to
automatic proof verification.
In particular, one of its variants is the axiom-schematization of first-order logic used in the
main database written in the Metamath language (see~\cite[Appendix C]{metamath}), which
formalizes ZFC set theory (and many other areas of mathematics) on top of it.
We give in Appendix~\ref{app:labels} a table of correspondence of scheme labels used in this
article and in that Metamath database, called \texttt{set.mm}, and browsable online at
\url{https://us.metamath.org/mpeuni/mmset.html}.
\end{remark}

\paragraph{Plan of the article}
In Section~\ref{sec:formal}, we define the formal systems that are the object of study of this paper.
In Section~\ref{sec:complete}, we recall the completeness results of Tarski and Megill.
In Section~\ref{sec:indep}, we prove the independence results stated above.
In particular, we exhibit a natural example in first-order logic of an independent scheme all
of whose object-instances are redundant.
Appendix~\ref{app:arithm} gives an elementary example of a similar phenomenon of independence
of a scheme all of whose object-like instances are redundant.
Appendix~\ref{app:variants} makes some comments on the various axiom schemes and on some variants
of~\texttt{TMM}.
Appendix~\ref{app:gen} gives an alternate proof due to Mario Carneiro of the independence
of~\texttt{gen}.
Appendix~\ref{app:labels} contains a table of correspondence of scheme labels used here and
in the Metamath database~\href{https://us.metamath.org/mpeuni/mmset.html}{\texttt{set.mm}}.

\paragraph{Acknowledgments}
I warmly thank Mario Carneiro and Norman Megill for careful reading of earlier drafts of this
article and very useful comments.
I also thank the contributors of the Metamath database \texttt{set.mm}, in particular
Norman Megill, Mario Carneiro, and Wolf Lammen for their work related to axiomatic questions.
I also thank the anonymous referee for corrections and suggestions of improvements.

\paragraph{Norman Megill}
I discovered the notion I called ``supertruth'' in August 2020 and rapidly began discussing
it with Mario Carneiro and Norman Megill.
Over the following year, it slowly extended to the present article, which benefited from many
exchanges with Mario and Norm.
Although I have never met Norm personally, I enjoyed many insightful discussions we had over the
years, on this and related topics.
The formalized mathematics tool Metamath is his creation, and the vibrant community he gathered
around it owes him a lot.
I dedicate this article to his memory.

\section{The formal system}
\label{sec:formal}
In this section, we define the formal systems that are the object of study of this paper.
They are special cases of ``Metamath systems\rlap{,}'' whose general definition can be found
in~\cite[Appendix~C.2]{metamath} and~\cite[Subsection~2.1]{carneiro}.

\subsection{Schemes}
\label{subsec:schemes}

We denote by $\omega = 0, 1, 2, \dots$ the set of natural numbers.
We define two disjoint sets of symbols: the sets of \define{variable metavariables} and the set
of \define{formula metavariables}, defined respectively by\footnote{More precisely, one can
consider two bijective functions $x$ and $\phi$ from $\omega$ onto two disjoint sets.}
\begin{align}
\vrMV &\coloneqq \{x_i \mid i \in \omega \},\\
\fmMV &\coloneqq \{\phi_i \mid i \in \omega \},
\end{align}
and we denote by
\begin{equation}
\MV \coloneqq \vrMV \cup \fmMV
\end{equation}
the set of \define{metavariables}.
We sometimes informally write $x, y, z, \dots$ instead of $x_0, x_1, x_2, \dots$ and
$\phi, \psi, \chi, \dots$ instead of $\phi_0, \phi_1, \phi_2, \dots$. 

A \define{language} is a set of symbols (disjoint from the above sets), whose members are called
\define{nonlogical predicates}, together with a function from this set to~$\omega$ called the
arity function of the language.
We will typically use the notation $\calL = \{P_0, P_1, P_2, \dots\}$, with an often implicit
arity function $r \colon \calL \to \omega$.

A \define{metaformula} (on a language $\calL$) is a finite string of characters in
$\{ \to , \neg, \forall, \equiv \} \cup \MV \cup \calL$ conforming to the usual formation
rules\rlap{.}\footnote{
Explicitly, a metaformula is either a formula metavariable, a predicate (equality or a
non-logical predicate) followed by the number of variable metavariables corresponding to its
arity, the universal quantifier followed by a variable metavariable and a metaformula, the
negation of a metaformula, or the implication of two metaformulas.}
The \define{height} of a metaformula is defined as usual.
For the sake of readability, we will use an infix notation, and therefore parentheses, but it is
understood that the formal language uses prefix notation, which renders parentheses unnecessary.
We denote the set of metaformulas by\footnote{We use the Kleene star to denote the set of finite
strings of characters on a given alphabet.}
\begin{equation}
\MF \subseteq \bigl( \{ \to , \neg, \forall, \equiv \} \cup \MV \cup \calL \bigr)^*
\end{equation}
or $\MF_\calL$ if emphasis on the language is needed.
Note that $\fmMV \subseteq \MF$.

A \define{scheme} is a triple consisting of a finite set of metaformulas (its hypotheses), a
metaformula (its conclusion), and a (finite) set of pairs of metavariables occurring in its
hypotheses or conclusion (its disjoint variable conditions, or \define{DV~conditions}; this
terminology will become clear in Equation~\eqref{eq:legitimateSubst}).
We use $\Phi, \Psi, \Chi, \dots$ to denote metaformulas or schemes (which one will be clear
from context).
If $\Phi$ is a metaformula or a scheme, then we denote by
\begin{equation}
\OC(\Phi) \subseteq \MV
\end{equation}
the set of metavariables \define{occurring} in it (for schemes, this means: occurring in its
hypotheses or conclusion).
We also define $\OC(x_i) \coloneqq \{x_i\}$ for $i \in \omega$.
If $\Phi$ is a scheme, then we denote by $\DV(\Phi)$ its set of DV~conditions.
With this notation, 
\begin{equation}
\DV(\Phi) \subseteq \calP_2(\OC(\Phi))
\end{equation}
for any scheme $\Phi$, where $\calP_2$ denotes the set of subsets of cardinality~2 of the given set.

A scheme will typically be written as $\Phi = (\{\Phi_1, \dots, \Phi_n\}, \Phi_0, \DV(\Phi))$ or
informally $\Phi_1 \andd \dots \andd \Phi_n \rimplies \Phi_0 \comma \DV(\Phi)$ as in the introduction.
When there is no ambiguity, especially when there are no hypotheses, we will generally use the
same notation for a scheme and its conclusion.
We will sometimes use informal self-explanatory notation.
For instance, the scheme $(\varnothing, \neg\forall x_0 \neg x_0 \equiv x_1, \{\{x_0, x_1\}\})$
may be abbreviated as $\exists x_0 x_0 \equiv x_1 \comma \DV(x_0, x_1)$.
We will often write $(\Phi_I, \Phi_0, D)$ instead of $(\Phi_I, \Phi_0, D \cap \calP_2(\OC(\Phi)))$.
Also, a scheme with no DV~conditions may be written simply as $(\Phi_I, \Phi_0)$
and a scheme with no hypotheses as $(\Phi_0, \DV(\Phi))$.

A (type-preserving) \define{substitution} is a function
\begin{equation}
\sigma \colon \MV \to \vrMV \cup \MF
\end{equation}
such that $\sigma(\vrMV) \subseteq \vrMV$ and $\sigma(\fmMV) \subseteq \MF$, and
$\{ m \in \MV \mid \sigma(m) \neq m \}$ is finite.
If a substitution is defined as a partial function, then it is assumed to be the identity where
not defined.
The action of a substitution on a metaformula is defined in the natural way as a simultaneous
substitution.
The resulting metaformula is called an ``instance'' of the original metaformula.
We denote the result of the action of the substitution~$\sigma$ on the metaformula~$\Phi$
by~$\Phi^\sigma$.
For example, the metaformula
$(\forall x_0 \phi_0)^{x_0 \leftarrow x_1, \phi_0 \leftarrow x_0\equiv x_2}$ is
$\forall x_1 x_0\equiv x_2$.

The action of a substitution on a scheme is defined as follows: the substitution acts on its
hypotheses and conclusion as above, and also on its set of DV~conditions in the natural way:
if $D \subseteq \calP_2(\MV)$ and $\sigma$ is a substitution, then
\begin{equation}
\label{eq:propagate}
D^\sigma \coloneqq \bigl\{
\{m, n\} \in \calP_2(\MV) \mid
\exists \{m', n'\} \in D \; (( m \in \OC(\sigma(m')) \wedge n \in \OC(\sigma(n')))
\bigr\}.
\end{equation}
We use the same notation as for metaformulas.
A substitution~$\sigma$ is \define{legitimate} on a scheme~$\Phi$ if it does not violate
its DV~conditions, that is, if
\begin{equation}
\label{eq:legitimateSubst}
\{m, n\} \in \DV(\Phi) \text{ implies }\OC(\sigma(m)) \cap \OC(\sigma(n)) = \varnothing.
\end{equation}
If a substitution is applied to a scheme, it will be implicitly assumed that it is legitimate
on that scheme.
An \define{instance} of a scheme is the result of a legitimate substitution possibly followed by
the addition of any number of DV~conditions.
Formally, $\Psi \in \Inst(\Phi)$ if and only if there exists a substitution~$\sigma$ legitimate
on~$\Phi$ such that $\Psi_i = {\Phi_i}^\sigma$ for $0 \leq i \leq n$ and
$\DV(\Phi)^\sigma \subseteq \DV(\Psi)$.
In this case, the substitution $\sigma$ is uniquely determined on $\OC(\Phi)$, and if we mention
``the'' substitution yielding $\Psi$, it will be the one equal to the identity outside
of~$\OC(\Phi)$.

Note the following important idempotence or transitivity property: an instance of an instance of a
metaformula or scheme is an instance of that metaformula or scheme (and a metaformula or scheme is
an instance of itself, so instantiation yields a preorder on the set of metaformulas and a preorder
on the set of schemes).

\subsection{Proofs}

Proofs are defined similarly as in classical logic.
The precise definition for general Metamath systems can be found in the cited
references\rlap{.}\footnote{
Actually, only the notion of provability is explicitly defined, though a definition of a proof
easily follows.}
A \define{proof} of a scheme~$\Phi$ from a set of schemes~$S$ is a couple consisting of~$\Phi$
and a finite sequence~$P$ of metaformulas (called the ``lines'' of the proof) satisfying the
following two conditions:
\begin{enumerate}
\item
if a line $P_i$ is not a hypothesis of $\Phi$, then there exist $i_1, \dots, i_n < i$ such that
$(\{P_{i_1}, \dots, P_{i_n}\}, P_i, \DV(P))$ is an instance of a scheme in $S$, where
\begin{equation}
\DV(P) \coloneqq \DV(\Phi) \cup \bigl\{ \{m, n\} \mid m \in \OC(P) \setminus \OC(\Phi)
\text{ and } n \in \OC(P) \bigr\}
\end{equation}
is the set of DV~conditions of $P$;
\item
the final line of the proof is the conclusion of $\Phi$.
\end{enumerate}
Note that $\DV(P) \cap \calP_2(\OC(\Phi)) = \DV(\Phi)$.
A scheme is \define{provable} from a set of schemes if there exists a proof of that scheme from
that set of schemes.
The set of schemes provable from the set of schemes $S$ is denoted by $\overline{S}$.
We also write $S \vdash \Phi$ for $\Phi \in \overline{S}$.
A metavariable in $\OC(P) \setminus \OC(\Phi)$ is called a \define{dummy} metavariable of~$P$,
and the above condition in the definition of $\DV(P)$ means that all dummy metavariables are
disjoint from all other metavariables.

We now define the action of a substitution on a proof.
Let $\sigma$ be a substitution and $P$ be (the second component of) a proof of the scheme $\Phi$.
First, rename every dummy metavariable of~$P$ to avoid clashes\rlap{.}\footnote{For instance,
rename each dummy metavariable $m_i$ of~$P$ to the metavariable $m_{i+N+1}$, where $N$ is the
largest $j \in \omega$ such that $\sigma(m_j) \neq m_j$ or $m_j \in \OC(\sigma(m_k))$ for
some~$k$ with $\sigma(m_k) \neq m_k$.}
Then, apply the substitution to each line of $P$.
The resulting sequence is denoted by $P^\sigma$ (it can be defined unambiguously as indicated in
the previous footnote).
Similarly, the action of a substitution on a set of schemes is defined to be the action on each
of its elements.

\begin{proposition}\label{prop:proof-instance}
A scheme proves all its instances.
If~$P$ is a proof of the scheme~$\Phi$ from the set of schemes~$S$ and $\sigma$ is a substitution,
then $P^\sigma$ is a proof of~$\Phi^\sigma$ from~$S^\sigma$ and from~$S$.
If a scheme is provable from a given set of schemes, then so are all its instances.
If the scheme~$\Phi$ is provable from the set of schemes~$S$ and every scheme in~$S$ is provable
from the set of schemes~$T$, then $\Phi$ is provable from~$T$.
\end{proposition}

\begin{proof}
The first claim is obvious: there exists a proof whose lines are precisely the hypotheses and
conclusion of that instance of the scheme.
The second claim is straightforward, and that a proof from~$S^\sigma$ is also a proof from~$S$
follows from transitivity of instantiation.
The third claim is a consequence of the first two.

As for the fourth claim, let~$P$ be a proof of~$\Phi$ from~$S$.
We can suppose, up to duplicating lines of $P$ starting from the last line and going backwards,
that a given line is used at most once as a hypothesis of a later line.
This corresponds to making the proof seen as a rooted directed acyclic graph (with the last line
as the root) into a directed tree.
We apply the following procedure starting from the last line of~$P$ and going backwards:
If $P_i$ is a hypothesis of $\Phi$, we do not modify it.
Else, there exist $i_1, \dots, i_n < i$ such that $(\{P_{i_1}, \dots, P_{i_n}\}, P_i, \DV(P))$
is an instance of a scheme in~$S$.
By the second claim, that scheme has a proof from~$T$, say~$Q$.
We can suppose that $Q$ also has the form of a directed tree, and we ``graft'' it in place of
$(\{P_{i_1}, \dots, P_{i_n}\}, P_i)$, that is, we rename dummy metavariables to avoid collisions,
we add to~$P$ the lines of~$Q$ other than $P_{i_1}, \dots, P_{i_n}, P_i$, preserving line order,
we update $\DV(P)$ with the new dummy metavariables, and we again duplicate the subtrees with
roots $P_{i_1}, \dots, P_{i_n}$ that are used more than once as hypotheses.
The added lines are the conclusions of instances of schemes of~$T$, and they are left untouched
by later stages of the procedure.
We repeat the procedure with the line which was originally~$P_{i-1}$.
The procedure terminates since the lines which are added at each step are left untouched.
The result of that procedure is a proof of~$\Phi$ from~$T$.
\end{proof}

We now give a sufficient condition for a set of schemes to be closed under provability.
This will be useful in proving independence results.

\begin{proposition}\label{prop:valid-sound}
Let $R$ be a set of hypothesis-free schemes.
Let $A \colon \Schemes \to \calP(\Schemes)$ be such that
$A \circ \Inst \subseteq A \subseteq \Inst$.
Then, the set
\begin{equation}\label{eq:validity}
S_{A, R} \coloneqq \left\{ \Phi \in \Schemes \mid
\forall \Psi \in A(\Phi)
\left(
\forall i \in [1, n] \; (\Psi_i, \DV(\Psi)) \in R
\Rightarrow
(\Psi_0, \DV(\Psi)) \in R
\right)
\right\}
\end{equation}
is closed under provability\rlap{.}\footnote{
As indicated in the previous section, we write schemes as triples
$\Psi = (\{\Psi_1, \ldots, \Psi_n\}, \Psi_0, \DV(\Psi))$ or as couples
$(\Psi_0, \DV(\Psi))$ when there are no hypotheses.}

If~$R$ is closed under instantiation, then $R = S_{\Inst,R} \cap \Schemeshypfree$.
\end{proposition}

\begin{proof}
Let $P$ be a proof of the scheme $\Phi$ from~$S_{A, R}$.
Let $\Psi \in A(\Phi)$.
Since $A \subseteq \Inst$, there exist a substitution~$\sigma$ and a set
$D \subseteq \calP_2(\OC(\Psi))$ such that
$\Psi = (\Phi^\sigma, \DV(\Phi^\sigma) \cup D)$.

We suppose that $(\Psi_i, \DV(\Psi)) \in R$ for all its hypotheses and we will prove by induction
that the hypothesis-free schemes $(P_i^\sigma, \DV(P^\sigma) \cup D)$ corresponding to each line of
$P^\sigma$, which is a proof of~$\Psi$ from schemes of~$S_{A, R}$ by
Proposition~\ref{prop:proof-instance}, are in $R$.
Since the last line corresponds to $(\Psi_0, \DV(\Psi))$, this will prove the proposition.

If the line $P_k$ is a hypothesis of $\Phi$, then $(P_k^\sigma, \DV(P^\sigma) \cup D)$ is some
$(\Psi_i, \DV(\Psi))$, which is in~$R$ by hypothesis.
Else, there exist $k_1, \dots, k_n < k$ such that $(\{P_{k_1}, \dots, P_{k_n}\}, P_k, \DV(P))$
is an instance of a scheme $\Chi \in S_{A, R}$.
Therefore, $(\{ P_{k_1}^\sigma, \dots, P_{k_n}^\sigma \}, P_k^\sigma, \DV(P^\sigma) \cup D)
\in A(\Inst(\Chi)) \subseteq A(\Chi)$.
By the induction hypothesis, its hypotheses are in~$R$, so it follows from $\Chi \in S_{A, R}$ that
its conclusion is also in~$R$.

Finally, $A \subseteq \Inst$ implies $A(\Schemeshypfree) \subseteq \calP(\Schemeshypfree)$, so 
$S_{A, R} \cap \Schemeshypfree = A^{-1}(\calP(R))$, and $\Inst^{-1}(\calP(R)) \subseteq R$ with
equality when~$R$ is closed under instantiation.
\end{proof}

Examples of maps~$A$ satisfying the hypotheses of that proposition are $\Inst$ itself, as well
as $\InstfmMVfree$, the map associating with each scheme the set of its formula-metavariable-free
instances.
In applications, $R$ will be a set of hypothesis-free schemes that we interpret as characterizing
a form of ``prevalidity'', from which we will define ``validity'' by Equation~\eqref{eq:validity}.

\subsection{From the scheme level to the object level (and back?)}

We now explain the relationship between the scheme level and the object level of usual logic. The
fundamental notion to connect these two levels is that of \define{object-instantiation}. We 
define a new set of symbols disjoint from $\MV$, the set of (individual) \define{variables},
by\footnote{
More precisely, one can consider a bijective function $v \colon \omega \to \VR$ with
$\MV \cap \VR = \varnothing$.}
\begin{equation}
\VR \coloneqq \{v_i \mid i \in \omega \}.
\end{equation}
Let $\calL_\text{obj}$ be a set, called the \define{object language}, such that
$\calL \subseteq \calL_\text{obj}$ (together with an arity function extending that of $\calL$).
The set of formulas of first-order logic is denoted by
$\FM \subseteq \bigl( \{ \to , \neg, \forall, \equiv \} \cup \VR \cup \calL_\text{obj} \bigr)^*$
and is defined as usual\rlap{.}\footnote{
There is no notion of disjoint variables or DV~condition at the object level: two variables $v_i$
and $v_j$ are simply the same if $i = j$ or different if $i\neq j$.}
Formulas will often be denoted by $\bm\phi, \bm\psi, \dots$ to differentiate them from formula
metavariables (elements of~$\fmMV$).
An object-substitution is a function $\tau \colon \MV \to \VR \cup \FM$ such that
$\tau(\vrMV) \subseteq \VR$ and $\tau(\fmMV) \subseteq \FM$.
The action of an object-substitution on a metaformula is defined in the natural way, and results in
a formula called an \define{object-instance} of the metaformula.
We use the same notation as for instantiation.
For example, the formula $(\forall x_0 \phi_0)^{x_0 \leftarrow v_1, \phi_0 \leftarrow v_0\equiv v_2}$
is $\forall v_1 v_0\equiv v_2$.

One extends this action to schemes as above, with the same notion of legitimacy (for example, if a
scheme has $\DV(x_0, \phi_0)$, then one cannot substitute in it $v_0 \equiv v_1$ for $\phi_0$
and $v_0$ for $x_0$). We call the resulting couples (consisting of a set of formulas and a formula)
\define{formulas with hypotheses}.

Similarly to the transitivity property of instantiation, an object-instance of an instance of a
scheme is an object-instance of that scheme.

Suppose that $P$ is a proof of a scheme $\Phi$ from a set of schemes $S$ and that $F$ is an
object-instance of~$\Phi$. It is immediate to construct from~$P$ a proof (at the object level) of~$F$
from object-instances of schemes in~$S$. A scheme is \define{object-provable} from a set of schemes
if all its object-instances are provable from object-instances of these schemes. Therefore, the above
shows that provability implies object-provability. Does the converse hold? Answer in the next
paragraph.

\begin{remark}
\label{rmk:language}
Object-provability depends on the language, whereas provability does not.
More precisely, let $\calL$ and $\calL'$ be two languages and let~$\Phi$ be a scheme and $S$ be a
set of schemes, both on the language $\calL \cap \calL'$.
Then, $\Phi$ is provable from $S$ in the language $\calL$ if and only if it is so in the
language~$\calL'$: in a proof on any language, replace all atomic expressions involving a nonlogical
predicate not in $\calL \cap \calL'$ by a constant, say $\top$.
The result is a proof of $\Phi$ from $S$ in $\calL \cap \calL'$.
On the other hand, it can be the case that $\Phi$ is object-provable from $S$ in a language but not
in a larger one, because object-instantiation involves the language: on a larger language, $\Phi$
has more object-instances.
This is detailed in the next paragraph where we recall Monk's result of object-independence
of~\texttt{subst} when the language has a non-nullary nonlogical predicate.
\end{remark}

The following proposition is the analogue at the object level of Proposition~\ref{prop:valid-sound},
and it has a very similar proof.

\begin{proposition}\label{prop:valid-sound-obj}
Let $R \subseteq \FM$ be a set of formulas.
Let $A \colon \Schemes \to \calP(\FMwithhyp)$ be such that
$A \circ \Inst \subseteq A \subseteq \operatorname{Obj-Inst}$.
Then, the set
\begin{equation}
T_{A, R} \coloneqq \{ \Phi \in \Schemes \mid
\forall \bm\phi \in A(\Phi) \;
(\forall i \in [1, n] \; \bm\phi_i \in R \Rightarrow \bm\phi_0 \in R) \}
\end{equation}
is closed under provability.
\end{proposition}

\paragraph{Independence}

A scheme is (\define{object}-)\define{independent from} a set of schemes if it is not (object-)provable
from them\rlap{.}\footnote{
Traditionally, independence also includes unprovability of the negation. In the cases under
consideration in this article, the negation is trivially unprovable since the considered scheme is
true in the standard model.}
Equivalently, a scheme is object-independent from a set of schemes if at least one of its
object-instances is not provable from their object-instances.
A scheme is (\define{object}-)\define{independent in} a set of schemes if it is (object-)independent
from the other schemes in that set.
A scheme is (\define{object}\nobreakdash-)\define{redundant in} a set of schemes if it is not
(object\nobreakdash-)independent
in that set, or equivalently if it is (object-)provable from the other schemes in that
set\rlap{.}\footnote{
We use ``redundant'' for ``not independent'' because ``dependent'' has another meaning.}

Since provability implies object-provability, object-independence implies independence.
The converse does not hold.
Examples can be given for Metamath systems that are more elementary than first-order logic, even if
somewhat artificial (see Appendix~\ref{app:arithm}).
In first-order logic, an example comes from Monk's proof of the object-independence of \texttt{subst}
in \texttt{TM} on any language containing a non-nullary nonlogical predicate, even though it is
object-redundant on an empty language, as proved by Tarski (\cite[Lem.~13 and~16]{tarski}).
The argument is as follows: since \texttt{subst} is object-independent on any language containing a
non-nullary nonlogical predicate (and no corresponding predicate axiom scheme), it is independent.
Therefore, if one adds the predicate axiom schemes associated with each predicate, one has
object-redundancy and independence.

Are there examples which do not require adding a nonlogical predicate and associated predicate
axiom schemes?
There are.
An example is given by our proof of the independence of~\texttt{spec} in~\texttt{TMM} in
Proposition~\ref{prop:spec}, even though it is object-provable from~\texttt{T} as proved by
Kalish--Montague.

\subsection{Emulating the object level at the scheme level}
\label{subsec:emulating}

The following remark due to Norman Megill is useful to treat some object-level problems at the
scheme level.

The mapping $\VR \to \MV, v_j \mapsto x_j$ gives rise to an injection
$i \colon \FM \hookrightarrow \MF$.
One can extend it to formulas with hypotheses as follows: $i$ acts as above on hypotheses and
conclusion, and adds all possible DV~conditions among occurring metavariables.
In particular, the image of~$i$ is the set of schemes containing no formula metavariables and with
all DV~conditions among occurring metavariables.
We call these schemes \textbf{object-like}.

Defining the action of $i$ on sets of formulas with hypotheses and on proofs in the natural way,
one sees that if $P$ is an object-level proof of the formula with hypotheses~$F$ from the set of
formulas with hypotheses~$S$, then $i(P)$ is a proof of the scheme $i(F)$ from the set of
schemes $i(S)$.
In particular, a scheme is object-provable from a set of schemes $S$ if and only if all its
object-like instances are provable from $S$.

\begin{remark}
To demonstrate the usefulness of this method, we prove that \texttt{spec} is object-provable from
\texttt{T}\rlap{.}\footnote{That fact is also a consequence of Theorem~\ref{thm:tarski}.}
We will actually prove the more general result: for all metaformula~$\Phi$ and for all variable
metavariable~$x$, one has
\begin{equation}
\label{eq:sp0}
\texttt{T} \vdash 
\Big(
\forall x \Phi \to \Phi
\comma
\big\{ \{x, m\} \mid m \in \OC(\Phi) \setminus \{x\} \big\}
\Big).
\end{equation}
This generalizes object-provability of \texttt{spec} in two ways: $\Phi$ may contain formula
metavariables, and the DV conditions are only with $x$, and not among other metavariables occurring
in~$\Phi$.
First, one has
\begin{equation}
\texttt{T} \vdash 
\Big(
x \equiv y \to (\phi \leftrightarrow \psi) \rimplies \forall x \phi \to \phi
\comma
\big\{\{x, y\}, \{x, \psi\}, \{y, \phi\}\big\}
\Big).
\end{equation}
This is proved as~\href{https://us.metamath.org/mpeuni/spw.html}{\texttt{spw}} in~\texttt{set.mm}.
The hypothesis of this scheme can be thought of as ``$\psi$ is the result of replacing every
occurrence of~$x$ in~$\phi$ by a fresh metavariable~$y$\rlap{.}''

Now, it suffices to prove the hypothesis of the above scheme for a particular choice of~$\psi$
(possibly with additional DV conditions as long as they do not exceed those of~\eqref{eq:sp0}).
We prove that for all metaformula~$\Phi$ and for all variable metavariables~$x$ and~$y$, if $y$ is
fresh (that is, if $y \notin \{ x\} \cup \OC(\Phi)$)\rlap{,}\footnote{If $x = y$, then the result
is a tautology (as for the weakened forms of \texttt{ALLcomm} and \texttt{ALLeq} below) but these
degenerate forms are not needed.} then
\begin{equation}
\texttt{T} \vdash
\Big(
x \equiv y \to (\Phi \leftrightarrow \Phi^{x \leftarrow y}),
\big\{ \{y, m\} \mid m \in \{x\} \cup \OC(\Phi)\big\} 
\Big).
\end{equation}
This is easily proved by induction on the height of~$\Phi$.
The base cases correspond to~$\Phi$ atomic, that is, of the form $\phi_i$, or $x_i \equiv x_j$, or
similarly with a nonlogical predicate.
These cases are respectively trivial and proved from~\texttt{EQ} and from the predicate axiom
schemes for the nonlogical predicates.
As for the induction steps, the propositional calculus cases (that is, $\Phi$ is an implication or
a negation) pose no difficulty.
If~$\Phi$ is of the form $\forall z \Psi$ with $x \neq z$, then the result follows easily
from~\texttt{ALLdistr}.
Finally, suppose that~$\Phi$ is of the form $\forall x \Psi$.
One has to prove
$\big( x \equiv y \to (\forall x \Psi \leftrightarrow \forall y \Psi^{x \leftarrow y}),
\big\{ \{y, m\} \mid m \in \OC(\Psi) \cup \{x\} \big\} 
\big)$ knowing by induction hypothesis that 
$\big( x \equiv y \to (\Psi \leftrightarrow \Psi^{x \leftarrow y}),
\big\{ \{y, m\} \mid m \in \OC(\Psi) \cup \{x\} \big\} 
\big)$.
This is a consequence of~\href{https://us.metamath.org/mpeuni/cbvalvw.html}{\texttt{cbvalvw}}
in~\texttt{set.mm}.

The provability result~\eqref{eq:sp0} is about the best one can hope for in~\texttt{T}, in terms of
DV~conditions, in view of the independence of $\forall x x \equiv y \to x \equiv y$ proved in the
proof of Proposition~\ref{prop:spec}.

Similarly, one can prove object-provability, and actually similar generalizations thereof as above,
corresponding to \texttt{modal5} (and similarly with \texttt{modalB} and
\texttt{modal4}\footnote{The axiom \texttt{modalD} is provable from \texttt{T} without need of
weakening.}), \texttt{ALLcomm}, \texttt{subst} and \texttt{ALLeq}:
for any metaformula~$\Phi$ and variable metavariables~$x$ and~$y$, one has, respectively,
\begin{gather}
\texttt{T} \vdash 
\Big(
\exists x \Phi \to \forall x \exists x \Phi
\comma
\big\{ \{x, m\} \mid m \in \OC(\Phi) \setminus \{x\} \big\}
\Big)\\
\texttt{T} \vdash 
\Big(
\forall x \forall y \Phi \to \forall y \forall x \Phi
\comma
\big\{ \{y, m\} \mid m \in \{x\} \cup \OC(\Phi) \setminus \{y\} \big\}
\Big)\\
\begin{multlined}
y \notin \{ x \} \cup \OC(\Phi), \texttt{T} \vdash 
\Big(
\forall x (x \equiv y \to (\Phi \to \forall x (x \equiv y \to \Phi)))
\comma\qquad\qquad\\\qquad\qquad
\big\{ \{x, m\} \mid m \in \OC(\Phi) \setminus \{x\} \big\} \cup
\big\{ \{y, m\} \mid m \in \{x\} \cup \OC(\Phi) \big\} 
\Big)
\end{multlined}\\
\texttt{T} \vdash 
\Big(
\forall x x \equiv y \to (\forall x \Phi \to \forall y \Phi)
\comma
\big\{ \{x, m\} \mid m \in \OC(\Phi) \setminus \{x\} \big\} \cup
\big\{ \{y, m\} \mid m \in \{x\} \cup \OC(\Phi) \setminus \{y\} \big\} 
\Big).
\end{gather}
\end{remark}

\section{Soundness and Completeness}
\label{sec:complete}

In order to define soundness and completeness of a system, one has to first define its
\define{semantics}.
At the object level, this is the usual semantics of first-order logic: a statement with hypotheses
is true\footnote{\label{footnote:valid}A more frequent term is ``(universally) valid\rlap{,}''
especially when ``true'' is reserved for specific models and assignments.} if it is true in every
nonempty model of first-order logic\footnote{We attribute to open formulas the truth value of their
universal closure; this is innocuous since we restrict ourselves to nonempty models.} (and in the
case of formulas with hypotheses, if it preserves truth in every such model).
At the scheme level, we define, as expected, a scheme to be \define{true} when all its
object-instances are true\rlap{.}\footnote{After all, this is the original aim of the introduction
of the scheme level: formally describe the object level. This differs from the scheme
semantics from~\cite{carneiro}, which considers more general models.}

\begin{theorem}[Soundness]
If a scheme is provable from a set of true schemes, then it is true.
\end{theorem}

\begin{proof}
Let $P$ be a proof of $\Phi$ from a set $S$ of true schemes and let $F$ be an object-instance of $\Phi$.
As mentioned in the previous section, one can construct from $P$ a proof (at the object level)
of~$F$ from object-instances of schemes in $S$.
These object-instances of true schemes are true formulas with hypotheses.
Therefore, by soundness of first-order logic, $F$ is true.
Since $F$ can be any object-instance of $\Phi$, this implies that $\Phi$ is true.
\end{proof}

The axiom schemes of \texttt{TMM} are true, so \texttt{TMM} and its subsystems can only prove true
schemes.

A set of schemes is called \define{complete} (resp.\ \define{object-complete}) if all true schemes
are provable (resp.\ object-provable) from it.
Since provability implies object-provability, completeness implies object-completeness.
We now restate in our language the two fundamental completeness theorems regarding the
systems\footnote{We use this term synonymously with ``set of schemes\rlap{,}'' now that the
background theory has been established.} considered here.

\begin{theorem}[{\cite[Thm.~5]{tarski}}]
\label{thm:tarski}
The system~\texttt{T} is object-complete.
\end{theorem}

\begin{theorem}[{\cite[Thm.~9.7]{megill}}]
The system~\texttt{TMM} is complete for schemes with no DV~conditions involving two formula metavariables.
\end{theorem}

\begin{remark}
Tarski proved his result with the scheme \texttt{spec} added to \texttt{T}, and Kalish--Montague
proved (\cite[Lem.~9]{km}) that \texttt{spec} is object-redundant, resulting in the present statement.
\end{remark}

\begin{remark}
The completeness results of Monk are a bit different, since he studies substitutionless logic.
Although he stays at the object level, his methods and results in the study of~\texttt{TM} can be
considered as an intermediate step between the object-completeness of~\texttt{T} and the
completeness of~\texttt{TMM}.
\end{remark}

As an aside, note that the system \texttt{TMM} contains three DV~conditions: one in \texttt{vacGen}
and two in \texttt{subst}. The DV~conditions in \texttt{subst} can be removed if one uses the
variant~\texttt{ax-12} of~\texttt{subst} described in Appendix~\ref{app:variants}, leaving only one
DV~condition in the axiom-schematization. The next proposition shows that this is optimal. It also
shows the necessity of the restriction ``with no DV~conditions involving two formula metavariables''
in Megill's completeness theorem.

\begin{proposition}\label{prop:dv}
An axiom-schematization of first-order logic that is complete (resp.\ complete for schemes with no
DV~conditions involving two formula metavariables, resp.\ with no DV~conditions involving formula
metavariables), contains at least one scheme that becomes false when removing its DV~conditions
involving two formula metavariables (resp.\ its DV~conditions involving formula metavariables,
resp.\ its DV~conditions).
\end{proposition}

\begin{proof}
Suppose that there exists a complete axiom-schematization with no scheme becoming false when
removing its DV~conditions involving two formula metavariables.
Such a system proves the scheme
$(\exists x \phi \to \forall x \phi) \vee (\exists x \psi \to \forall x \psi) \comma \DV(\phi, \psi)$.
Indeed, this scheme is true (if $\phi$ and $\psi$ are disjoint, then at least one of them is
disjoint from~$x$).
Therefore, that same scheme without its DV~condition would also be provable: simply remove all
DV~conditions involving two formula  metavariables from its proof (this is still a proof because
DV~conditions involving two formula metavariables cannot be produced from DV~conditions of other
forms by Equation~\eqref{eq:propagate}).
But that scheme is false: for instance, substitute $x\equiv y$ for both $\phi$ and $\psi$.

For the second (resp.\ third) case, proceed similarly using the scheme \texttt{vacGen}
(resp.\ the scheme \texttt{oneObj}, see next remark).
\end{proof}

\begin{remark}
A result of Monk (\cite{monk2}) implies that there is no finite axiom-schematization with no
DV~conditions that is object-complete (let alone complete for all schemes involving no DV~conditions).
Indeed, when there are no DV~conditions in our schemes, our notion of object-instantiation reduces
to that article's notions of substitution and instantiation, so its main theorem proving the
non-existence of finite axiomatizations applies.
\end{remark}

\begin{remark}
On the other hand, there exist finite object-complete axiom-schematizations of first-order logic
with no DV~conditions involving formula metavariables.
An example can be obtained from \texttt{TMM} by replacing~\texttt{vacGen} with the axiom scheme
\begin{equation}
\forall x x \equiv y \to (\phi \to \forall x \phi) \comma \DV(x, y) \tag{\texttt{oneObj}}
\end{equation}
and adding for each nonlogical predicate~$P$ an axiom scheme analogous to \texttt{genEq}:
\begin{equation*}
\neg \forall x x \equiv x_1 \to ( \dots \to (\neg \forall x x \equiv x_n \to (P(x_1, \dots, x_n) \to \forall x P(x_1, \dots, x_n))) \dots) \tag{\texttt{gen$_P$}}.
\end{equation*}
We prove, following Megill, that this axiom-schematization is object-complete.
It suffices to prove that all object-like instances of \texttt{vacGen} are provable from it.
We do this by induction on the height of $\Phi$.
If $\Phi$ is an atomic formula involving the predicate $P$, then it is easily proved from
$\texttt{intuitcalc} \cup \{ \texttt{oneObj}, \texttt{gen}_P \}$ (this includes the equality
predicate since we use $\texttt{gen}_\equiv$ as a synonym of \texttt{genEq}).
The induction step in the case of an implication, a negation, or a universal quantification is
easily proved from \texttt{pure}\footnote{See
\href{https://us.metamath.org/mpeuni/hbim.html}{\texttt{hbim}},
\href{https://us.metamath.org/mpeuni/hbn.html}{\texttt{hbn}}, and
\href{https://us.metamath.org/mpeuni/hbal.html}{\texttt{hbal}} in \texttt{set.mm}, where the base
cases are given by
\href{https://us.metamath.org/mpeuni/ax5eq.html}{\texttt{ax5eq}} and
\href{https://us.metamath.org/mpeuni/ax5el.html}{\texttt{ax5el}}.} (see definition of this subsystem
in Figure~\ref{fig:subsystems} of Appendix~\ref{app:variants}).
\end{remark}

\begin{remark}
These results leave open the questions:
\begin{itemize}
\item
Are there finite axiom-schematizations with no DV~conditions involving formula metavariables that
are complete for all schemes of the same form?
(The previous remark shows that there are object-complete ones.)
\item
Are there finite axiom-schematizations that are complete for all schemes (including those with
DV~conditions involving two formula metavariables)?
\end{itemize}
\end{remark}

\section{Independence}
\label{sec:indep}

As mentioned above, since provability implies object-provability, object-independence (on any
given language) implies independence.
We begin by giving the easier object-independence results.
For some results, we suppose that the language contains a non-nullary nonlogical predicate, and
object-independence in the empty language remains an interesting open question.

\subsection{Object-independence}

The proofs in this subsection are at the object level, so we use valuations defined on formulas (and
not metaformulas), that is, functions $\val \colon \FM \to A$ where $A$ is a set of truth values,
which will be here of the form $\{0, 1, \ldots, n\}$.
A subset of these values are considered as true.
Valuations will be implicitly defined by induction on formula height.

\subsubsection{Propositional calculus}

\paragraph{Object-independence of modus ponens}

We actually have the stronger result that there is no object-complete finite axiom-schematization of
first-order logic such that the only axiom scheme with hypotheses is the rule of generalization.
Indeed, consider the formulas
$\bm\phi^n \coloneqq
(x_0 \equiv x_1 \to (x_1 \equiv x_2 \to \ldots (x_{n-1} \equiv x_n \to x_0 \equiv x_n)))$
for $n \geq 1$, expressing a general form of transitivity of equality.
Since they do not contain universal quantifiers, they are provable only if they are
object-instances of axiom schemes.
Suppose $\bm\phi^n \in \operatorname{Obj-Inst}(\Phi)$.
Then, we have $\Phi = (\Phi_0 \to (\Phi_1 \to \ldots (\Phi_{i-1} \to \Phi_i)))$ where
$0 \leq i \leq n $ and each~$\Phi_j$ with
$j < i$ is either a formula metavariable or is of the form $x_{a_j} \equiv x_{b_j}$, and~$\Phi_i$
is either a formula metavariable different from all its antecedents or is of the form
$x_c \equiv x_d$, in which case $i = n$.
The only allowed equality among the $x$-indices are $b_j = a_{i+1}$ and $c = a_0$ and $d = b_i$.
A straightforward verification shows that the only valid such scheme is, up to variable
metavariable renaming,
$x_0 \equiv x_1 \to (x_1 \equiv x_2 \to \ldots (x_{n-1} \equiv x_n \to x_0 \equiv x_n))$.
This leads to the need of infinitely many axiom schemes.

Another proof, which can be adapted to other systems, is to note that an axiom-schematization
as above would lead to a decidable set of universally valid formulas, which is not the case.

\paragraph{Object-independence of the axioms of propositional calculus}

If one ignores quantifiers, then the axiom schemes of \texttt{TMM} not in propositional calculus
nor the equality bloc become particularly simple.
Formally, ``ignoring quantifiers'' means choosing valuations $\val$ such that
$\val(\forall v_i \bm\phi) = \val(\bm\phi)$ for any $i \in \omega$ and any formula $\bm\phi \in \FM$.
We use such valuations to prove the object-independence of the propositional calculus axiom schemes.
We will define these valuations in terms of assignments of variables to elements of a structure,
$a \colon \VR \to X$.
In practice, only the \define{tautologous assignment} $a \colon \VR \to \omega, v_i \mapsto i$ will
be used to witness the counterexamples below.\footnote{
The extra step of using assignments is not necessary to proving object-independence (and the definitions
of the valuations below could be simplified accordingly), but it makes the counterexamples more readily
applicable to other contexts of deductive systems with substitution rules.}

\paragraph{Object-independence of \texttt{minimp}}

Consider the valuation given by the truth table
\begin{center}
\begin{tabular}{|c|ccccc|c|}
\hline
$\to$&0&1&2&3&4&$\neg$\\
\hline
$^*$0\phantom{$^*$}&0&1&1&1&1&2\\
1&0&0&0&0&0&0\\
2&0&0&0&0&0&0\\
3&0&0&4&0&4&1\\
4&0&0&3&3&0&1\\
\hline
\end{tabular}
\end{center}
where only~0 is considered true\rlap{,}\footnote{This table was found by computer search to solve a
similar problem in~\cite{jubin}.} and, for any assignment of variables $a\colon \VR \to X$, set
$\val(v_i \equiv v_j) \coloneqq 0 \text{ if $a(v_i) = a(v_j)$ else } 3$.
This validates \texttt{mp}, \texttt{notnotintro}, \texttt{simp}, and \texttt{id} (see
Appendix~\ref{app:labels} for the labels) for all assignments~$a$, and therefore all the axiom schemes not
in propositional calculus, including the axiom schemes in the equality bloc, and also \texttt{contrap},
\texttt{notelim}, and \texttt{peirce}, but falsifies the object-instance of \texttt{minimp} given by
$\phi \leftarrow v_0 \equiv v_0,
 \psi \leftarrow v_0 \equiv v_1,
 \chi \leftarrow v_0 \equiv v_0,
 \theta \leftarrow v_0 \equiv v_0,
 \tau \leftarrow \neg v_0 \equiv v_0$,
as seen using the tautologous assignment.

\paragraph{Object-independence of \texttt{peirce}}

Use the G\"odel (2)truth-table from~\cite{robinson}
\begin{center}
\begin{tabular}{|c|ccc|c|}
\hline
$\to$&0&1&2&$\neg$\\
\hline
$^*$0\phantom{$^*$}&0&1&2&2\\
1&0&0&2&2\\
2&0&0&0&0\\
\hline
\end{tabular}
\end{center}
where only~0 is considered true, and set
$\val(v_i \equiv v_j) \coloneqq 0 \text{ if $a(v_i) = a(v_j)$ else } 1$.
It validates intuitionistic propositional calculus, hence also the later axiom schemes, including
the axiom schemes in the equality bloc, but falsifies the object-instance of \texttt{peirce} given
by $\phi \leftarrow v_0 \equiv v_1, \psi \leftarrow \neg v_0 \equiv v_0$.

\paragraph{Object-independence of \texttt{contrap}}

Use the truth-table
\begin{center}
\begin{tabular}{|c|ccc|c|}
\hline
$\to$&0&1&2&$\neg$\\
\hline
$^*$0\phantom{$^*$}&0&1&1&1\\
1&0&0&0&0\\
2&0&0&0&1\\
\hline
\end{tabular}
\end{center}
where only 0 is considered true, and set
$\val(v_i \equiv v_j) \coloneqq 0 \text{ if $a(v_i) = a(v_j)$ else } 2$.
It validates implicational calculus and \texttt{notnotintro}, and therefore all the axiom schemes
not in propositional calculus.
It validates the axiom schemes in the equality bloc, and also \texttt{notelim}, but falsifies the
object-instance of \texttt{contrap} given by
$\phi \leftarrow v_0 \equiv v_1, \psi \leftarrow v_0 \equiv v_0$.

\quad

From this point on, valuations have to satisfy implicational calculus (that is,
$\{\texttt{mp}, \texttt{minimp}, \texttt{peirce}\}$).
Therefore, we make the following definitions.
If $P$ is a proposition (in the informal language), then\footnote{This notation is called the
Iverson bracket, see D.\ E.\ Knuth, \textit{Two notes on notation}, Amer. Math. Monthly 99 (1992),
no. 5, 403--422, \href{https://arxiv.org/abs/math/9205211}{arXiv:math/9205211}[math.HO].}
$[P]$ is equal to 1 if $P$ is true, else 0.
A valuation $\val \colon \FM \to \{0, 1\}$ is an \define{imp-valuation} if
$\val(\bm\phi \to \bm\psi) = \max (1 - \val(\bm\phi), \val(\bm\psi))$ for all formulas
$\bm\phi, \bm\psi \in \FM$.
It is a \define{pc-valuation} if furthermore $\val(\neg \bm\phi) = 1 - \val(\bm\phi)$ for all
formulas $\bm\phi \in \FM$.
An \define{imp$_\equiv$-valuation} (resp.\ a \define{pc$_\equiv$-valuation}) is an imp-valuation
(resp.\ a pc-valuation) that is \define{standard on equality}, that is, such that
$\val(v_i \equiv v_j) = [i = j]$.\footnote{For quantifier-free formulas, this valuation
corresponds to satisfaction in the structure~$\omega$ for the tautologous assignment.}

\paragraph{Object-independence of \texttt{notelim}}

Use the imp$_\equiv$-valuation which ignores quantifiers and is always true on negations, that is,
such that $\val(\neg \bm\phi) = 1$ for all~$\bm\phi \in \FM$.
Then, all axiom schemes are validated except~\texttt{notelim}: consider its instance given by
$\phi \leftarrow v_0 \equiv v_0, \psi \leftarrow v_0 \equiv v_1$.

\subsubsection{Modal bloc}

\paragraph{Object-independence of generalization}

Because \texttt{subst} is a generalization of a scheme over a variable that is free in it, a proof
of the independence of \texttt{gen} is a bit harder to find than in~\cite{km} or~\cite{monk}.
Note also that $\{\texttt{propcalc}, \texttt{EQrefl}, \texttt{genEq}\} \vdash \forall x x \equiv x$.

Since \texttt{propcalc}, \texttt{EQ} and \texttt{spec} have to hold, we define $\val$ to be a
pc$_\equiv$-valuation\footnote{%
A priori, we could allow $\val$ to interpret $\equiv$ like any equivalence relation on the domain
of discourse (with the assignment from the previous footnote), but in order that it validate
\texttt{subst}, it is nearly constrained to interpret $\equiv$ as equality on the domain of
discourse (with the assignment from the previous footnote).} given on quantified formulas by
\begin{equation}
\val(\forall v_i \bm\phi ) \coloneqq \val(\bm\phi) \val_i(\bm\phi)
\end{equation}
for all $i \in \omega$ and $\bm\phi \in \FM$, where the $\val_i$'s are valuations.
Since $\val$ has to satisfy \texttt{ALLdistr}, we assume that each $\val_i$ is an imp-valuation.
We also assume that $\val_i$ ignores universal quantifiers, so that \texttt{ALLcomm} is validated.
More precisely, we set $\val_i(\forall v_j \bm\phi) \coloneqq \val_i(\bm\phi)$ if $i \neq j$ and
$\val_i(\forall v_i \bm\phi) \coloneqq 1$.
On atomic formulas, we define
\begin{equation}
\val_i(v_j \equiv v_k ) \coloneqq
[j = k \text{ or } i \notin \{j, k\}]
\end{equation}
for $i, j, k \in \omega$.
Finally, $\val_i$ is defined on negated formulas by:
\begin{align*}
\val_i(\neg (\bm\phi \to \bm\psi)) &\coloneqq  \val_i(\bm\phi \to \bm\psi),\\
\val_i(\neg\neg \bm\phi) &\coloneqq  \val_i(\neg \bm\phi),\\
\val_i(\neg v_j\equiv v_k) &\coloneqq  \val_i(v_j\equiv v_k),\\
\val_i(\neg \forall v_j \bm\phi) &\coloneqq 1
\end{align*}
for $i, j, k \in \omega$ and $\bm\phi, \bm\psi \in \FM$.
In other words, $\val_i$ ignores negations not followed by a universal quantifier and validates
negations of universally quantified formulas.

Since $\val$ is a pc-valuation, it validates \texttt{propcalc}.
Since $\val$ and each $\val_i$ are imp-valuations, $\val$ validates \texttt{ALLdistr}.
Since $\val(\forall v_i \bm\phi) \leq \val(\bm\phi)$, $\val$ validates \texttt{spec}.
Since $\val$ is standard on equality, it validates \texttt{EQ}.

Let $\FV(\bm\phi)$ denote the set of free variables occurring in the formula $\bm\phi$.
A proof by induction on formula height shows that if $v_i \notin \FV(\bm\phi)$, then
$\val_i(\bm\phi) = 1$.
Therefore, $\val$ validates \texttt{modal5} and \texttt{vacGen}.
One has $\val(\forall v_i \forall v_j \bm\phi) = \val(\bm\phi) \val_i(\bm\phi) \val_j(\bm\phi)$
if $i \neq j$, which is symmetric in $i, j$.
Therefore, $\val$ validates \texttt{ALLcomm}.
One has
$\val(\forall v_j \neg v_j \equiv v_i) = \val(\neg v_j \equiv v_i) \val_j(v_j \equiv v_i) = 0$
(even when $i=j$).
Therefore, $\val$ satisfies \texttt{denot} (even when its DV~condition is dropped).

If $i \neq j$, then $\val(v_i \equiv v_j) = \val_i(v_i \equiv v_j) = 0$, and $\val_i$ is an
imp-valuation, so $\val$ validates \texttt{subst}.
Also, $\val$ validates \texttt{genEq}, since it actually validates
$v_j \equiv v_k \to \forall v_i v_j \equiv v_k$.
It also validates the strong denotation axiom $\neg \forall v_j \neg v_j \equiv v_i$
and all its generalizations.
Finally, $\val$ validates \texttt{ALLeq} since $\val(\forall v_i v_i \equiv v_j) = [i = j]$.

To show that $\val$ does not validate \texttt{gen}, note that
$\val(\neg v_0 \equiv v_1) = 1 - [0 = 1] = 1$ and $\val_0(\neg v_0 \equiv v_1) = [0 = 1] = 0$
so $\val(\forall v_0 \neg v_0 \equiv v_1) = 0$.

Another proof of independence (but not object-independence), due to Mario Carneiro, is given in
Appendix~\ref{app:gen}.

\paragraph{Independence of \texttt{ALLdistr}}

One can prove the scheme $\forall x (\phi \to \psi) \andd \forall x \phi \rimplies \forall x \psi$
from $\{ \texttt{mp}, \texttt{gen}, \texttt{spec} \}$.
Therefore, the independence of \texttt{ALLdistr} cannot be proved by simple valuations.

We prove the object-independence of \texttt{ALLdistr} in $\texttt{TMM} \setminus \{ \texttt{spec} \}$
using Monk's valuation (\cite[Thm.~9, Part~4 (C4)]{monk}).
Let $\val$ be the pc-valuation evaluating atomic formulas to 1 and such that
$\val(\forall v_i (\bm\phi \to \bm\psi)) = 1$ for all $\bm\phi, \bm\psi \in \FM$ and
$\val(\forall v_i \bm\phi) = \val(\bm\phi)$ for all $\bm\phi \in \FM$ that is not an implication.
Then, $\val$ validates all the axiom schemes of $\texttt{TMM} \setminus \{\texttt{spec}\}$ except
\texttt{ALLdistr}: take for instance $x \leftarrow v_0$, $\phi \leftarrow v_0 \equiv v_0$,
$\psi \leftarrow \neg v_0 \equiv v_0$.

We also prove the object-independence of \texttt{ALLdistr} in
$\texttt{TMM} \setminus \{ \texttt{vacGen}, \texttt{subst} \}$ in any language containing at least
two nonlogical predicates (of any arity) as follows.
Recall that a \emph{neighborhood model} for a classical non-normal modal logic is given by a set
of worlds~$W$ with a neighborhood function $N \colon W \to \calP \calP W$ and a valuation~$\vDash$
satisfying propositional calculus and such that $w \vDash \square \bm\phi$ if and only if
$\{ v \in W \mid v \vDash \bm\phi \} \in N(w)$.
After translating first-order formulas into modal ones as usual (map universal quantifiers to
necessity), consider a neighborhood model $(W, N)$ where
$W = \{ w_1, w_2, w_3 \}$ and
$N(w_1) = \{ \{w_1\}, \{ w_1, w_2 \}, W \}$ and
$N(w_2) = \{ \{w_1, w_2\}, \{ w_2, w_3 \}, W \}$ and
$N(w_3) = \{ \{w_3\}, \{ w_1, w_3 \}, \{ w_2, w_3 \}, W \}$.
Let $P$ and $Q$ be two distinct fixed nonlogical predicates of the language.
Consider a valuation such that $P$ is false exactly at $w_3$ and $Q$ is true exactly at~$w_2$
(regardless of their arguments).
Then, $(P \to Q)$ is true exactly at $\{ w_2, w_3 \}$, so
$w_2 \vDash \forall v_0 (P \to Q)$ and
$w_2 \vDash \forall v_0 P$ but
$w_2 \not\vDash \forall v_0 Q$.
On the other hand, $W$ is a neighborhood of each world, so the model validates the rule of
generalization, and the model is such that $w \in U$ for all $w \in W$ and $U \in N(w)$, so
it validates \texttt{spec}, and is such that $U \notin N(v)$ implies
$\{ w \in W \mid U \notin N(w) \} \in N(v)$, so it validates \texttt{modal5}.
Finally, take this valuation to validate every equality everywhere.

Finally, we prove the independence of \texttt{ALLdistr} in \texttt{TMM}.
We consider the languages $\calL = \varnothing$ and $\calL_\text{obj} = \{P, Q \}$.
These two predicates are evaluated as in the previous paragraph and we denote by $R$
the set of true formulas for the corresponding valuation.
We define the function $A \colon \Schemes \to \calP(\FMwithhyp)$ by
\[
A(\Phi) \coloneqq \left\{ \tau(\Phi) \in \operatorname{Obj-Inst}(\Phi) \middle|
P, Q \notin \tau(\bigcup \DV(\Phi))
\right\}
\]
for any $\Phi \in \Schemes$.
By $P \notin \tau(\bigcup \DV(\Phi))$, we mean that $P$ does not occur in the formulas $\tau(\phi)$
with $\phi \in \bigcup \DV(\Phi)$ (we do not use the function $\OC$, which we defined only for
metavariable occurrences, not for occurrences of any symbols).

Using the notation of Proposition~\ref{prop:valid-sound-obj}, we consider as ``true'' the schemes
in $T_{A, R}$.
In order to prove that that set is closed under provability, by Proposition~\ref{prop:valid-sound-obj},
it suffices to prove that $A \circ \Inst \subseteq A$.

Let $\Psi \in A(\Inst(\Phi))$.
There exists an object-instantiation $\tau$ such that $\Psi = \tau(\Chi)$ with $\Chi \in \Inst(\Phi)$
and $P, Q \notin \tau(\bigcup \DV(\Chi))$.
Since $\Chi \in \Inst(\Phi)$, there exist $\sigma, D$ such that
$\Chi = (\Phi^\sigma, \DV(\Phi^\sigma) \cup D)$.
We want to prove that $\Psi \in A(\Phi)$.
Consider the object-instantiation $\tau' \coloneqq \tau \circ \sigma$.
It suffices to prove that $P, Q \notin \tau'(\bigcup \DV(\Phi))$.
Note that $\tau(\sigma(\bigcup \DV(\Phi))) = \tau(\OC(\sigma(\bigcup \DV(\Phi)))) \cup
\bigl( \sigma(\bigcup \DV(\Phi)) \setminus \OC(\sigma(\bigcup \DV(\Phi))) \bigr)$,
so it suffices to prove that $P, Q$ are in none of those two sets.
First, $P, Q \notin \calL$ implies $P, Q \notin \sigma(\bigcup \DV(\Phi))$.
Second, by Equation~\eqref{eq:propagate}, we have
$\OC(\sigma(\bigcup \DV(\Phi))) = \bigcup \DV(\Phi^\sigma)$.
Therefore, $P, Q \notin \tau(\bigcup \DV(\Chi)) \supseteq \tau(\bigcup \DV(\Phi^\sigma))
= \tau(\OC(\sigma(\bigcup \DV(\Phi))))$.

In that model, \texttt{ALLdistr} is still not satisfied, all the other axiom schemes in
$\texttt{TMM} \setminus \{\texttt{vacGen}, \texttt{subst}\}$ are still satisfied, and now both
\texttt{vacGen} and \texttt{subst} are also satisfied: indeed, any $\bm\phi \in A(\texttt{vacGen})$
is of the form $\bm\psi \to \forall v_i \bm\psi$ with $P, Q \notin \OC(\bm\psi)$, so
$w_i \Vdash \bm\psi$ for either all worlds or for none, and similarly for \texttt{subst}.

\paragraph{Object-independence of \texttt{spec}}

The axiom scheme \texttt{spec} is object-provable from~\texttt{T} (since~\texttt{T} is
object-complete) but independent in~\texttt{TMM}.

\begin{remark}
The pc-valuation which is false on every atomic formula and true on every universally quantified
formula validates all axiom schemes in
$\texttt{TMM} \cup \{ \forall x x \equiv x \} \setminus \{ \texttt{EQrefl} \}$ except \texttt{spec},
proving object-independence of \texttt{spec} in the version of \texttt{TMM} where \texttt{EQrefl} is
replaced by its universal closure.
\end{remark}

We prove the object-independence of \texttt{spec} in $\texttt{TMM} \setminus \{\texttt{ALLeq}\}$
when the language
has a non-nullary nonlogical predicate (but no associated predicate axiom scheme).
This implies independence when the language has a non-nullary nonlogical predicate, hence also
when it does not by Remark~\ref{rmk:language}.
Alternatively, we could also consider that predicate to be in $\calL_\text{obj} \setminus \calL$.
Consider a model of first-order logic without equality, with domain $D = \{ a, b\}$ and
interpret~$\equiv$ as the total relation (that is, its graph is~$D^2$).
Consider a unary predicate~$P$ which is interpreted to true on~$a$ and false on~$b$.
Finally, interpret ``$\forall v_i$'' as ``$\forall v_i \in \{a\}$'' (more precisely, modify the
standard interpretation of universally quantified formulas accordingly).
Then, all the axiom schemes of $\texttt{TMM} \setminus \{\texttt{ALLeq}\}$ are true except
\texttt{spec}: its instance $\forall v_0 P(v_0) \to P(v_0)$ is false since the~$v_0$ in the
consequent can be assigned to~$b$.

To rephrase this example, we can say that we have a \emph{domain of quantification}~$\{a\}$ which
is strictly included in the \emph{domain of discourse}~$D$, so that specialization does not hold.
In order that \texttt{denot} hold, we need that all elements in the domain of discourse be equal
to an element in the domain of quantification, which is why we interpreted~$\equiv$ as the total
relation.
Note that \texttt{ax-12} is not true in this model (this is necessary, since one can prove
specialization from~\texttt{T} and \texttt{ax-12}), although \texttt{subst}~is.

In the next section, we give a proof of the independence of \texttt{spec}, by proving independence
of its instance $\forall x x \equiv y \to x \equiv y$, in~\texttt{TMM}.

\paragraph{Independence of \texttt{modal5}}

The axiom scheme \texttt{modal5} is object-provable from~\texttt{T} (since~\texttt{T} is
object-complete) but independent in $\texttt{TMM}$, as a standard Kripke model (see for
instance~\cite{modal}) followed by a restricted object-instantiation rule shows: consider two
worlds, $A$ and $B$, with a reflexive accessibility relation which furthermore connects $A$ to $B$.
Interpret equality as always true and ``$\forall v_i$'' for any $i \in \omega$ as necessity.
Introduce a predicate $P$ (at the object level) which is true in $A$ but not in $B$.
Then, $A \not\Vdash \exists v_i P \to \forall v_i \exists v_i P$, so \texttt{modal5} does not hold,
but the other axioms in $\texttt{TMM} \setminus \{\texttt{vacGen}, \texttt{subst}\}$ do.
This proves object-independence in $\texttt{TMM} \setminus \{\texttt{vacGen}, \texttt{subst}\}$.
Let $R$ be the set of formulas which are true in that sense.

Similarly as in the proof of independence of \texttt{ALLdistr}, consider the languages
$\calL = \varnothing$ and $\calL_\text{obj} = \{P \}$ and
define the function $A \colon \Schemes \to \calP(\FMwithhyp)$ by
\[
A(\Phi) \coloneqq \left\{ \tau(\Phi) \in \operatorname{Obj-Inst}(\Phi) \middle|
P \notin \tau(\bigcup \DV(\Phi))
\right\}
\]
for any $\Phi \in \Schemes$.
As in the proof of independence of \texttt{ALLdistr}, we use Proposition~\ref{prop:valid-sound-obj},
whose hypotheses are proved to hold in a similar way as above, and declare ``true'' the schemes
in $T_{A, R}$.

In that model, \texttt{modal5} is still not satisfied, all the other axiom schemes in
$\texttt{TMM} \setminus \{\texttt{vacGen}, \texttt{subst}\}$ are still satisfied, and now both
\texttt{vacGen} and \texttt{subst} are also satisfied, for the same reason as in the proof of
independence of \texttt{ALLdistr}.

\subsubsection{Vacuous generalization}

The axiom scheme \texttt{vacGen} is independent in \texttt{TMM} by
Proposition~\ref{prop:dv}\rlap{.}\footnote{
Since \texttt{subst} is still true without the condition $\DV(y, \phi)$, see for instance the proof
\url{https://us.metamath.org/mpeuni/bj-ax12.html}.}
Equivalently, declare a scheme to be $*$-true if the scheme obtained from it by ignoring its
DV~conditions involving formula metavariables is true.
Then, \texttt{vacGen} is the only axiom scheme in \texttt{TMM} which is not $*$-true.
We leave the question of object-independence in~\texttt{TMM} open\rlap{.}\footnote{The comment at
\url{https://us.metamath.org/mpeuni/ax5ALT.html} mentions its object-redundancy, but this is
relative to a larger axiom-schematization.}

\subsubsection{Equality bloc}

We consider the following models of first-order logic without equality:
\begin{itemize}
\item
\texttt{EQrefl}: a non-empty domain, equality is interpreted as the empty relation
(so $v_i \equiv v_j$ is always evaluated to false).
\item
\texttt{EQsymm}: domain $\{0, 1\}$, the graph of equality is $\{0, 1\}^2 \setminus \{(1, 0)\}$.
\item
\texttt{EQtrans}: domain $\{0, 1, 2\}$, the graph of equality is
$\{0, 1, 2\}^2 \setminus \{(0, 2), (2, 0)\}$.
\end{itemize}
The first model proves object-independence of \texttt{EQrefl} in \texttt{TMM}.
The other two prove object-independence  of \texttt{EQsymm} and of \texttt{EQtrans} in
$\texttt{TMM} \setminus \{ \texttt{subst}, \texttt{Alleq}\}$.

We can also prove object-independence of \texttt{EQsymm} and of \texttt{EQtrans} in
$\texttt{TMM} \setminus \{ \texttt{denot} \}$ by considering respectively the following two
pc-valuations ignoring quantifiers: let $\val(v_i \equiv v_j) = [i \leq j]$ for the first, and
$\val(v_i \equiv v_j) = [|i-j| \leq 1]$ for the second.

\subsubsection{Denotation axiom scheme}

Let $\val$ be the pc$_\equiv$-valuation that ignores quantifiers.
Then, all axiom schemes are validated except for \texttt{denot}.
This proves its object-independence in \texttt{TMM}.

\subsubsection{Substitution axiom scheme}

The model used in~\cite[Thm.~9, Part~8 (C8)]{monk} proves the independence of \texttt{subst} in a
system related to~\texttt{TMM}.
We specialize that model to prove independence in~\texttt{TMM}, even when the latter is augmented
with the axiom schemes \texttt{oneObj} and $\texttt{gen}_P$
(but without the predicate axiom scheme associated with $P$)
(see Appendix~\ref{app:labels} for these names).

The language consists of one nonlogical predicate, say~$P$, which is unary.
The domain of discourse has three objects, say 0, 1, 2.
Equality is interpreted as the equivalence relation with equivalence classes $\{0, 1\}$ and $\{2\}$.
The predicate~$P$ holds for exactly~0.

Having more than one equivalence class for $\equiv$ (\textit{i.e.,} having unequal objects) makes
the formula $\forall v_i v_i \equiv v_j$ true exactly when $i = j$, making the verifications easier.

\subsubsection{Predicate axiom schemes}

Suppose that the language has $n$ nonlogical predicate symbols, $\calL = \{ P_1, \ldots, P_n \}$
with $P_i$ of arity~$a_i$.
There is no loss of generality since only a finite number of predicates can occur in any proof.
The corresponding $\sum_{i=1}^n a_i$ predicate axiom schemes have the form
\begin{equation}
y \equiv z \to (P_i(x_1, \ldots, y, \ldots, x_{a_i}) \to P_i(x_1, \ldots, z, \ldots, x_{a_i}))
\tag{\texttt{ax-${P_i}_j$}}.
\end{equation}
where $y$ and $z$ are at the $j^{\mathrm{th}}$ position, for $1 \leq i \leq n$ and $1 \leq j\leq a_i$.

Fix $i$ and $j$ as above
and consider the pc-valuation $\val$ which ignores quantifiers and is always true on equalities
(that is, $\val(v_k \equiv v_l) = 1$ for any $k, l \in \omega$), and such that
$\val(P_k(v_{s(1)}, \ldots, v_{s(a_i)})) = [k = i \text{ and } s(j) = 0]$.
This proves object-independence of \texttt{ax-${P_i}_j$} in \texttt{TMM}.

\subsection{Supertruth and partial independence of \texttt{ALLcomm}}

Since the system \texttt{T} is object-complete, proving independence of the axiom schemes in
$\texttt{TMM}\setminus\texttt{T}$ is generally a harder task.
In the cases of \texttt{spec} and \texttt{subst}, this was done by adding nonlogical predicates
to the language without the associated predicate axiom schemes.
The method presented here is different, and can prove, for instance, that even some true
formula-metavariable-free schemes, like the instance $\forall x x \equiv y \to x \equiv y$ of
\texttt{spec}, or the instance $\forall x \forall y z \equiv t \to \forall y \forall x z \equiv t$
of \texttt{ALLcomm}, are not provable from \texttt{T}.

\begin{remark}
Monk's proof of the independence of \texttt{spec} in a related system
(\cite[Thm.~12, Part~3 (A6)]{monk}) does not apply here.
Indeed, that system contains only sentences (closed formulas), and for instance \texttt{denot}
is evaluated to false (its closure is of course evaluated to true).
His proof of the independence of \texttt{ALLcomm} (Part~1 (A4) of the same theorem) relies on
the order of variables, because he works in a substitutionless calculus, but that method cannot
be applied here, where variable metavariables are interchangeable.
\end{remark}

To ease the reading of the proof, we introduce the following definitions and notation.
A \define{quantified subformula} is a subformula\footnote{More properly, ``submetaformula\rlap{.}''}
beginning with (and not merely containing) a universal quantifier.
If it begins with ``$\forall x_i$'' then we say that it is \define{$i$-quantified}.
The operation on metaformulas that consists in replacing every occurrence of $x_j$ within any
$i$-quantified subformula (equivalently, within the scope of any quantification over $x_i$) with
$x_i$ is called the \define{$(i, j)$-transform}.
The \define{$(i, j)$-transform of a scheme} is the scheme resulting from the $(i, j)$-transforms of
its hypotheses and conclusion, with its DV~conditions unchanged\rlap{.}\footnote{Except of course
removing the DV~conditions involving $x_j$ if it does not occur in the resulting scheme.}

The result of the $(i, j)$-transform on the metaformula or scheme $\Phi$ is denoted by $\Phi^{(i, j)}$.
If $\Phi^{(i, j)} = \Phi$, we say that the $(i, j)$-transform acts trivially, or is \define{trivial}, on~$\Phi$.
This is the case for instance when $\{x_i, x_j\} \notin \calP_2(\OC(\Phi))$.
The $(i, j)$-transform is \define{legitimate} on the scheme $\Phi$ if $\{x_i, x_j\} \notin \DV(\Phi)$\rlap{.}\footnote{Note that an $(i, i)$-transform is legitimate and trivial on any scheme.}

\begin{remark}
One can think of an $(i, j)$-transform as a way to set quantified subformulas to True or False with more freedom than with usual models.
In the words of Mario Carneiro, it can be thought of as a ``deliberate bound variable capture\rlap{.}''
\end{remark}

Using the notation from Proposition~\ref{prop:valid-sound}, we denote by~$R$ the set of
hypothesis-free schemes such that all the legitimate $(i, j)$-transforms of all their instances
are true, and we call \define{supertrue} the schemes in~$S_{\Inst, R}$.
Since~$R$ is closed under instantiations, Proposition~\ref{prop:valid-sound} gives
$R = S_{\Inst, R} \cap \Schemeshypfree$.

\begin{remark}
A few remarks are in order:
\begin{itemize}
\item
When a hypothesis or the conclusion of a scheme is considered as a scheme itself (for example when asserting or asking about its truth or supertruth), then it is understood to be the scheme with no hypotheses, with that formula as conclusion, and with the DV~conditions inherited from the scheme.
For example, if $(\{\Phi_1, \Phi_2\}, \Phi_0, D)$ is a scheme, and we mention ``the hypothesis $\Phi_1$ as a scheme'', then we mean, the scheme $(\varnothing, \Phi_1, D)$ (by our above convention, we do not need to write $D \cap \mathcal{P}_2(\OC(\Phi_1))$ explicitly).

\item
Recall that we do include the trivial $(i, j)$-transforms.
In other words, ``all legitimate $(i, j)$-transforms of $\Phi$'' is the same thing as ``$\Phi$ and all its legitimate $(i, j)$-transforms\rlap{.}''

\item
The definition may seem to have one unnecessary level of instantiation (namely, two levels instead of one), but this is subtly not the case: the first instantiation is of the whole scheme, while the second is, independently, of the given hypothesis or conclusion.
As an example, consider the scheme $\phi \rimplies \psi$.
If we did not have the first instantiation, then the definition would read ``if all transforms of all instances of $\phi$ are true, then all transforms of all instances of $\psi$ are true'', and this is vacuously true since obviously not all instances of $\phi$ are true ($\bot$ is not true).
Therefore, the scheme $\phi \rimplies \psi$ would be supertrue, clearly something that we do not want.
With the correct definition, we are allowed to first instantiate the scheme, and one such instance is $\top \rimplies \psi$.
Now, all transforms of all instances of $\top$ are true, but this is not the case for $\psi$ (again, $\bot$ is not true).
Therefore, $\phi \rimplies \psi$ is not supertrue.

\item
For hypothesis-free schemes, however, one can remove one level of instantiation, since, as noted above,
$R = S_{\Inst, R} \cap \Schemeshypfree$.
Namely, a hypothesis-free scheme is supertrue if and only if all the legitimate $(i, j)$-transforms of all its instances are true.
Combined with the first item, this shows that a scheme is supertrue if and only if all its instances are such that whenever their hypotheses are supertrue, so is their conclusion.

\item
One could be tempted to simplify the definition to: ``a scheme is supertrue if all the legitimate $(i, j)$-transforms of all its instances are true\rlap{.}''
That property implies supertruth, but not conversely: with that definition, the rule of generalization would not be supertrue as the following example due to Mario Carneiro shows: the rule of generalization instantiates to $(x \equiv y \to \phi) \rimplies \forall x (x \equiv y \to \phi)$, which $(x, y)$-transforms to $(x \equiv y \to \phi) \rimplies \forall x (x \equiv x \to \phi)$, which is not true: consider the instance given by $\phi \leftarrow x \equiv y$.

\item
For any metaformula or scheme $\Phi$, one has $\OC(\Phi^{(i, j)}) \subseteq \OC(\Phi)$.

\item
For any metaformula or scheme $\Phi$, one has $\Phi^{x_j \leftarrow x_i} = \bigl(\Phi^{(i, j)}\bigr)^{x_j \leftarrow x_i}$, so  $\Phi^{x_j \leftarrow x_i}$ is an instance of $\Phi^{(i, j)}$.
\end{itemize}
\end{remark}

\begin{proposition}\label{prop:prov}
A scheme provable from a set of supertrue schemes is supertrue.
\end{proposition}

\begin{proof}
This follows from Proposition~\ref{prop:valid-sound} applied with $A = \Inst$ and $R$ the set of hypothesis-free schemes all legitimate $(i, j)$-transforms of instances of which are true.
\end{proof}

\begin{lemma}
A true quantifier-free scheme is supertrue.
\end{lemma}

\begin{proof}
Let $\Phi$ be a quantifier-free scheme and let $\Psi$ be an instance of $\Phi$.
Let $i, j \in \omega$.
Since $\Phi$ is quantifier-free, the $i$-quantified subformulas of $\Psi$ necessarily come from
formula metavariables $\phi_k$ occurring in~$\Phi$.
Therefore, $\Psi^{(i, j)}$ is also an instance of $\Phi$: if $\Psi$ was obtained by substituting
$\Chi$ for $\phi_k$, then ${\Psi}^{(i, j)}$ is obtained by substituting $\Chi^{(i, j)}$ for $\phi_k$.
Therefore, if $\Phi$ is true and all hypotheses of $\Psi$ are supertrue, then its conclusion $\Psi_0$ is supertrue.
Therefore, if $\Phi$ is true, then it is supertrue.
\end{proof}

\begin{lemma}
The rule of generalization \texttt{gen} is supertrue.
\end{lemma}

\begin{proof}
Let $\Phi$ be a metaformula and let $i, j, k \in \omega$.
One has $(\forall x_k \Phi)^{(k, j)} = \forall x_k \Phi^{x_j \leftarrow x_k}$.
If $i \neq k$, then $(\forall x_k \Phi)^{(i, j)} = \forall x_k \Phi^{(i, j)}$.
Since an instance of \texttt{gen} has the form $(\{\Phi\}, \forall x_k \Phi, D)$, the above proves that \texttt{gen} is supertrue.
\end{proof}

\begin{lemma}
The schemes \texttt{ALLdistr}, \texttt{modalD}, \texttt{modal4}, \texttt{modal5}, \texttt{vacGen}, \texttt{denot}, \texttt{subst}, \texttt{genEq} are supertrue.
\end{lemma}

\begin{proof}
\texttt{ALLdistr}, \texttt{modalD}, \texttt{modal4}, \texttt{modal5}:
An $i$-quantified subformula in an instance of one of these schemes with $i \neq 0$ comes from within a formula metavariable, so as in the case of quantifier-free schemes, any of its $(i, j)$-transforms with $i \neq 0$ is an instance of the scheme.
On the other hand, the result of a $(0, j)$-transform is also an instance, obtained by the substitution $x_j \leftarrow x_0$.

\texttt{vacGen}:
Let $\Phi$ be an instance of \texttt{vacGen}.
It is of the form $\Psi \to \forall x_k \Psi \comma \DV(x_k, \OC(\Psi))$.
If the $(i, j)$-transform is nontrivial on $\Phi$, then $j \in \OC(\Psi)$.
Therefore, if it is also legitimate, then $i \neq k$ because of the DV~conditions.
Therefore, any $i$-quantified subformula of $\Phi$ is a subformula of $\Psi$.
Therefore, $\Phi^{(i, j)} = (\Psi^{(i, j)} \to \forall x_k \Psi^{(i, j)})$.
Since $\OC(\Psi^{(i, j)}) \subseteq \OC(\Psi)$, we obtain that the $(i, j)$-transform of $\Phi$ is an instance of \texttt{vacGen}.
Therefore, all $(i, j)$-transforms of all instances of \texttt{vacGen} are instances of \texttt{vacGen} so are true, so \texttt{vacGen} is supertrue.

\texttt{denot}:
This scheme has no legitimate nontrivial $(i, j)$-transforms.
Its strengthening obtained by removing its DV~condition has a legitimate nontrivial
$(i, j)$-transform, which is $x \equiv x \to \exists x x \equiv x$, which is true.

\texttt{subst}:
Let $\Phi$ be an instance of \texttt{subst}.
Up to variable renaming, it is characterized by the formula $\Psi$ that is substituted for $\phi$.
All legitimate $(i, j)$-transforms act similarly on both occurrences of $\Psi$, so that the
resulting scheme $\Phi^{(i, j)}$ is an instance of \texttt{subst} so is true.

\texttt{genEq}:
The only nontrivial $(i, j)$-transforms are the $(x, y)$-transform and the $(x, z)$-transform, and
both make an antecedent false.
\end{proof}

\begin{proposition}\label{prop:ALLcomm}
In the axiom system $\texttt{TMM} \setminus \{\texttt{spec}, \texttt{ALLeq} \}$, the
scheme~\texttt{ALLcomm} cannot be weakened by adding the DV~condition $\DV(x, y)$ or $\DV(y, \phi)$;
in particular, it is independent.
\end{proposition}

\begin{proof}
An instance of \texttt{ALLcomm} is of the form
$\Psi \coloneqq \forall x \forall y \Phi \to \forall y \forall x \Phi$ (with possibly DV conditions)
where $\Phi \in \MF$.
Its $(i, j)$-transforms other than the $(x, y)$-transform and the $(y, x)$-transform affect
both occurrences of $\Phi$ similarly, so they are instances of \texttt{ALLcomm}, hence are true.
Its $(y, x)$-transform is
$\forall x \forall y \Phi^{x \leftarrow y} \to \forall y \forall y \Phi^{x \leftarrow y}$,
which is essentially an instance of \texttt{spec}, hence is true.
Its $(x, y)$-transform is
$\forall x \forall x \Phi^{y \leftarrow x} \to \forall y \forall x \Phi^{y \leftarrow x}$.

An instance of $(\texttt{ALLcomm}, \DV(x, y))$ is of the form $(\Psi, \DV(x, y))$ with~$\Psi$
as above.
Since the $(x, y)$-transform is illegitimate on it, $(\texttt{ALLcomm}, \DV(x, y))$ is supertrue.
An instance of $(\texttt{ALLcomm}, \DV(y, \phi))$ is of the form
$(\Psi, \{ \{y, m \} \mid m \in \OC(\Phi) \})$ with~$\Psi$ as above.
Its $(x, y)$-transform (legitimate or not) is essentially generalization over a nonfree variable,
hence is true.
Therefore, $(\texttt{ALLcomm}, \DV(y, \phi))$ is supertrue.

On the other hand, the instance
$(\forall x \forall y z \equiv t \to \forall y \forall x z \equiv t, \{\{x, z\}, \{x, t\}\})$
of $(\texttt{ALLcomm}, \DV(x, \phi))$ is not supertrue since its $(x, y)$-transform is
$(\forall x \forall x z \equiv t \to \forall y \forall x z \equiv t, \{\{x, z\}, \{x, t\}\})$,
which is not true, as its instance $z \leftarrow y$ shows.

By the above paragraphs, the three previous lemmas, and Proposition~\ref{prop:prov}, all schemes
provable from $\texttt{TMM} \setminus \{\texttt{spec}, \texttt{ALLeq}, \texttt{ALLcomm} \}$ and
the above two weakenings of \texttt{ALLcomm} are supertrue.
Since \texttt{ALLcomm} is not, it is not provable from them.
\end{proof}

\subsection{Variants of supertruth, partial independence of \texttt{ALLeq}, and independence of  \texttt{spec}}

\paragraph{Hull-supertruth}

A first variant is the following notion: define the hull of any set of schemes to be the smallest
set of schemes containing it and closed under instantiation and legitimate $(i, j)$-transforms,
and then define a scheme to be hull-supertrue if for all of its instances, if all the schemes in
the hull of its hypotheses are true, then all the schemes in the hull of its conclusion are true.
Hull-supertruth is stronger for hypothesis-free schemes, but need not be for general schemes.
All the supertrue axiom schemes of \texttt{TMM} are hull-supertrue.

\paragraph{``Disjointing supertruth''}

We define the $\pp{i}{j}$-transform to be the $(i, j)$-transform followed by adding all DV conditions,
that is, $\Phi^{\pp{i}{j}} \coloneqq (\Phi^{(i, j)}, \calP_2(\MV))$.
With this version, all supertrue axiom schemes remain supertrue: this is trivial for hypothesis-free
axiom schemes, and straightforward for the two rules of inference.
The scheme \texttt{ALLcomm} is now valid with this notion: by the proof of
Proposition~\ref{prop:ALLcomm}, it suffices to prove that the $\pp{x}{y}$-transform
$(\forall x \forall x \Phi^{y \leftarrow x} \to \forall y \forall x \Phi^{y \leftarrow x}, \calP_2(\MV))$
of the instance $\Psi$ of \texttt{ALLcomm} is true, and this is the case since it is essentially an
instance of \texttt{vacGen}.

An instance of \texttt{ALLeq} is of the form
$\Psi \coloneqq \forall x x \equiv y \to (\forall x \Phi \to \forall y \Phi)$ (with possibly DV conditions)
where $\Psi \in \Schemes$.
Its $\pp{i}{j}$-transforms other than the $\pp{x}{y}$-transform are of the form
$(\forall x x \equiv y \to (\forall x \Phi^{(i, j)} \to \forall y \Phi^{(i, j)}), \calP_2(\MV))$,
which is true (the antecedent is false in models with at least two elements and the
consequent is true in models with one element).
Its $\pp{x}{y}$-transform is
$(\forall x x \equiv x \to (\forall x \Phi^{y \leftarrow x} \to \forall y \Phi^{(x,y)}), \calP_2(\MV))$.

Since the $\pp{x}{y}$-transform is illegitimate on $(\texttt{ALLeq}, \DV(x, y))$, that scheme is valid.
An instance of $(\texttt{ALLeq}, \DV(y, \phi))$ is of the form
$(\Psi, \{ \{y, m \} \mid m \in \OC(\Phi) \})$ with~$\Phi, \Psi$ as above and $y \notin \OC(\Phi)$.
The latter condition implies $\Phi^{y \leftarrow x} = \Phi^{(x,y)} = \Phi$, so its $\pp{x}{y}$-transform
is a consequence of $\{ \texttt{propcalc}, \texttt{spec}, \texttt{vacGen}\}$.
An instance of $(\texttt{ALLeq}, \DV(x, \phi))$ is of the form
$(\Psi, \{ \{x, m \} \mid m \in \OC(\Phi) \})$ with~$\Phi, \Psi$ as above and $x \notin \OC(\Phi)$.
The $\pp{x}{y}$-transform is legitimate only if $y \notin \OC(\Phi)$, and is then
$(\forall x x \equiv x \to (\forall x \Phi \to \forall y \Phi), \calP_2(\MV))$, which is a
consequence of $\{ \texttt{propcalc}, \texttt{spec}, \texttt{vacGen}\}$.

On the other hand, \texttt{ALLeq} is not valid: the $\pp{x}{y}$-transform of its instance
$\forall x x \equiv y \to (\forall x x \equiv y \to \forall y x \equiv y)$ is
$\forall x x \equiv x \to (\forall x x \equiv x \to \forall y x \equiv y), \{\{x, y\}\})$,
which is not true.
This proves:

\begin{proposition}
\label{prop:ALLeq}
In the axiom system $\texttt{TMM} \setminus \{\texttt{spec} \}$, the scheme~\texttt{ALLeq} cannot
be weakened by adding any DV~condition;
in particular, it is independent.
\end{proposition}

\paragraph{Semisupertruth}

We define the \define{$\{i, j\}$-transform} to be the combination of the $(i, j)$-transform
and the $(j, i)$-transform.
It is legitimate on a scheme when any (hence both) of these is legitimate.
Performing these two transforms in different orders gives schemes which are identical up to
renaming some bound variables, but for the sake of definiteness, one can ask that the
$(i, j)$-transform is performed first if $i < j$.
If one simply replaces $(i, j)$-transforms with $\{i, j\}$-transforms in the definition of
supertruth, there are already some differences.
For instance, the scheme $\forall x x \equiv y \to \forall y x \equiv y$
(\href{https://us.metamath.org/mpeuni/ax-c11n.html}{\texttt{ax-c11n}} in \texttt{set.mm}) is not
supertrue but satisfies this new notion.

We define the $\bb{i}{j}$-transform to be the $\{i, j\}$-transform followed by adding all
DV conditions.
Finally, we define a scheme to be \define{semisupertrue} if it is true and the following is true of
all its instances: if all the legitimate $\bb{i}{j}$-transforms of all the instances of its
hypotheses are true, then all the legitimate $\bb{i}{j}$-transforms of all the instances of
its conclusion are true.

All axiom schemes in $\texttt{TMM} \setminus \{\texttt{spec}\}$ are semisupertrue.
This is proved as above for axiom schemes in $\texttt{TMM} \setminus \{\texttt{ALLcomm}, \texttt{ALLeq}, \texttt{spec}\}$.
The only non-trivial $\bb{i}{j}$-transform of an instance of \texttt{ALLcomm} is of the form
$(\forall x \forall x \Phi^{y\leftarrow x} \to \forall y \forall y \Phi^{x\leftarrow y}, \calP_2(\MV))$,
which is true since $x$ is required to be disjoint from any metavariable in $\Phi^{x\leftarrow y}$
and $y$ from any metavariable from $\Phi^{y\leftarrow x}$.
The only non-trivial $\bb{i}{j}$-transform of an instance of \texttt{ALLeq} is of the form
$(\forall x x \equiv x \to (\forall x \Phi^{y \leftarrow x} \to \forall y \Phi^{x \leftarrow y},
\{\{x, y\}\}$) which is true for the same reason.

The instance of \texttt{spec} $\forall x x \equiv y \to x \equiv y$ is not semisupertrue since
its $\bb{i}{j}$-transform is $(\forall x x \equiv x \to x \equiv y, \{\{x, y\}\})$, which is not true.
However, the weakening of \texttt{spec} $(\forall x \phi \to \phi, \{\{x, \phi\}\})$ is semisupertrue.
This proves:

\begin{proposition}
\label{prop:spec}
In the axiom system \texttt{TMM}, the scheme~\texttt{spec} cannot
be weakened by adding any DV~condition; in particular, it is independent.
\end{proposition}

\begin{remark}
In this section, I used supertruth and some of its variants as devices to prove independence results.
I do not know whether they are of independent interest.
I leave their semantic study (for instance, giving an axiom-schematization for supertrue statements)
to future work.
\end{remark}

\appendix

\section{An elementary example of an independent but object-provable scheme}
\label{app:arithm}

In this appendix, we present an example of a Metamath system in which a scheme is independent
but all its variable-free instances are provable from a set of schemes. This is only
superficially similar to the example in the main part of the article, since here there is no
distinction between scheme level and object level.

The expressions are formed from the constants $\Term, \Nat, 0$, and~$'$ (prime).
We denote by $n$ a variable.
We posit that every variable, 0, and the prime of a term are terms.
We posit that 0 is a Nat, and that if $n$ is a $\Nat$, then so is $n'$.
Symbolically, our axiom schemes are:
\begin{gather*}
\Term n\\
\Term 0\\
\Term n'\\
\Nat 0\\
\Nat n \rimplies \Nat n'
\end{gather*}
The first three axioms can be considered as syntactic axiom schemes.
In this system, the scheme $\Nat n$ is independent although all its variable-free instances,
like $\Nat 0'$ or $\Nat 0''$, are provable.
More formally, the following is a valid Metamath database:
\begin{quote}
\begin{verbatim}
$c Term Nat 0 ' $.
$v n $.
on $f Term n $.
o0 $a Term 0 $.
os $a Term n ' $.
n0 $a Nat 0 $.
${ ns.1 $e Nat n $. ns $a Nat n ' $. $}
n1 $p Nat 0 ' $= o0 n0 ns $.
n2 $p Nat 0 ' ' $= o0 os n1 ns $.
nn $p Nat n $= ? $. $( not provable $)
\end{verbatim}
\end{quote}

The reason why the scheme $\Nat n$ is independent is that the system lacks a sort of induction
axiom scheme. This is similar to Robinson's arithmetic Q, where commutativity of addition is not
provable although any specific instance of it is.

\section{Comments and variants of the \texttt{TMM} system}
\label{app:variants}

Before mentioning the variants, we first show in Figure~\ref{fig:subsystems} various subsets
of the set of axiom schemes of \texttt{TMM} and the logics they axiomatize. The abbreviations
should be self-explanatory (``minimal implicational calculus\rlap{,}'' ``paraconsistent
calculus\rlap{,}'' etc.).

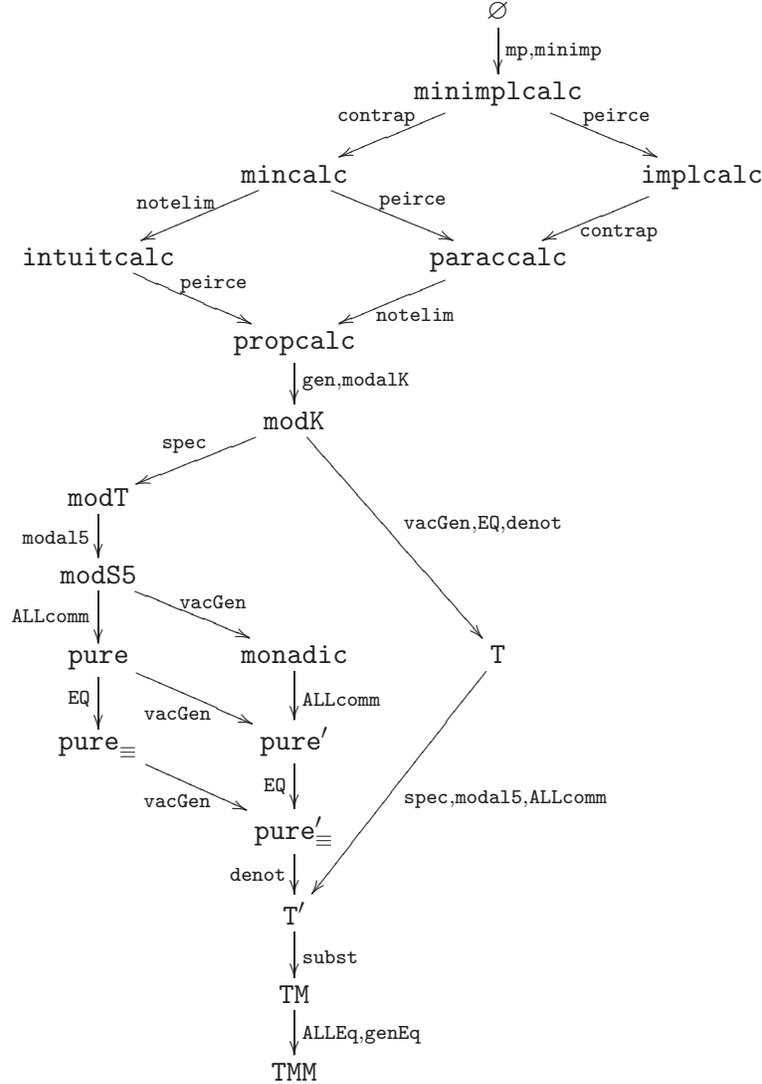
\begin{figure}[!ht]
\begin{equation*}
\xymatrix@=1.5em{
&& \varnothing \ar[d]^{\texttt{mp}, \texttt{minimp}}\\
&& \texttt{minimplcalc} \ar[dr]^{\texttt{peirce}} \ar[dl]_{\texttt{contrap}} \\
& \texttt{mincalc} \ar[dr]^{\texttt{peirce}} \ar[dl]_{\texttt{notelim}}
 && \texttt{implcalc} \ar[dl]^{\texttt{contrap}}\\
\texttt{intuitcalc} \ar[dr]^{\texttt{peirce}} && \texttt{paraccalc} \ar[dl]^{\texttt{notelim}}\\
& \texttt{propcalc} \ar[d]^{\texttt{gen}, \texttt{modalK}}\\
& \texttt{modK} \ar[dl]_{\texttt{spec}} \ar[dddr]^{\texttt{vacGen}, \texttt{EQ}, \texttt{denot}}\\
\texttt{modT} \ar[d]_{\texttt{modal5}} \\
\texttt{modS5} \ar[d]_{\texttt{ALLcomm}} \ar[dr]^{\texttt{vacGen}}\\
\texttt{pure} \ar[d]_{\texttt{EQ}} \ar[dr]_{\texttt{vacGen}}
 & \texttt{monadic} \ar[d]^{\texttt{ALLcomm}}
 & \texttt{T} \ar[dddl]^{\texttt{spec}, \texttt{modal5}, \texttt{ALLcomm}}\\
\texttt{pure}_\equiv \ar[dr]_{\texttt{vacGen}} &\texttt{pure}' \ar[d]_{\texttt{EQ}}\\
&\texttt{pure}_\equiv' \ar[d]_{\texttt{denot}}\\
& \texttt{T}' \ar[d]^{\texttt{subst}}\\
& \texttt{TM} \ar[d]^{\texttt{ALLeq}, \texttt{genEq}}\\
& \texttt{TMM}
}
\end{equation*}
\caption{Some subsystems of first-order logic}
\label{fig:subsystems}
\end{figure}

For the subsystems of propositional calculus in that figure, we refer to~\cite{robinson}.
Beware that the metalogic has no notion of free and bound variables, so the logics labeled
\texttt{monadic}, \texttt{pure}, and their variants may be weaker than expected.
In particular, in $\texttt{pure}_\equiv'$, the symbol $\equiv$ denotes an equivalence relation,
but not necessarily equality: it is equality only once \texttt{denot} has been assumed.

Some of the comments below relate the axioms of~\texttt{TMM} with the axioms used in the
Metamath database \href{https://us.metamath.org/mpeuni/mmset.html}{\texttt{set.mm}}.
The statement labeled \texttt{xxx} in that database will typically be denoted by
\texttt{set.mm/xxx}.
A table of correspondence can be found in Table~\ref{tab:labels} of Appendix~\ref{app:labels}.

\subsection{The propositional calculus bloc}

The sole inference rule of the propositional calculus part of all the variants mentioned here
is modus ponens \texttt{mp}.
A family of variants axiomatizes minimal implicational calculus differently than with the sole
scheme \texttt{minimp} (due to {\L}ukasiewicz).
One can for instance use the axiomatization \texttt{SK} (as done in \texttt{set.mm})
or \texttt{BCKW} or $\texttt{B}'\texttt{KW}$\rlap{.}\footnote{
As usual, we name implicational calculus tautologies after the combinator they correspond to.}
Although these systems are known to be independent, it is not the case anymore for \texttt{BCKW}
and $\texttt{B}'\texttt{KW}$ when one adds the Peirce axiom~\texttt{P}, since by a classical
result of {\L}ukasiewicz, both \texttt{BCP} and $\texttt{B}'\texttt{KP}$ already axiomatize
implicational calculus.
As for \texttt{SK}, the truth table used above to prove independence of \texttt{minimp} also
proves independence of~\texttt{S} in that axiomatization; a computer search for truth tables
might answer the question of independence of~\texttt{K}, but some later axioms become instances
of the identity axiom~\texttt{I} and of~\texttt{K} when quantifiers are ignored, so models may be
harder to find.

Independently of the above, one can replace the set
$\{\texttt{peirce}, \texttt{contrap}, \texttt{notelim}\}$ with the sole scheme
$(\neg \phi \to \neg \psi) \to (\psi \to \phi)$ (as done in \texttt{set.mm}) and the independence
results are obviously preserved.
One can also replace $\{\texttt{peirce}, \texttt{notelim}\}$ with $\neg \neg \phi \to \phi$,
and the independence results are obviously preserved as well.

Tarski used the axiomatization of propositional calculus
$\{\texttt{syl}, \texttt{clavius}, \texttt{notelim}\}$, and Kalish--Montague and Monk proved their
independence in the larger system they study (but their proofs do not apply to \texttt{TMM} since
some later axiom schemes are not satisfied by the models they use).

There are many other known independent axiomatizations of propositional calculus, but proofs of
independence in the larger system \texttt{TMM} have to be found.

The axiomatization we use, inspired by~\cite{robinson}, has the advantage of having semantic
significance: six of its eight subsystems obtained from using \texttt{minimp} and a combination of
the other three axioms correspond to a well-studied logic: (minimal implicational, minimal,
implicational, intuitionistic, paraconsistent, classical)-calculus.

\subsection{The modal bloc}

The three modal axioms of our system are $\texttt{ALLdistr} = \texttt{modalK}$,
$\texttt{spec} = \texttt{modalT}$, and $\texttt{modal5}$.
Using the language of modal logic, one can replace the bloc\footnote{%
As is customary in modal logic, we denote by ``T5'' the conjunction of the modal axioms~T and~5,
which are denoted here by \texttt{modalT} and \texttt{modal5} respectively, and similarly for
similar expressions.} T5 with DB4 (or TB4).
The scheme \texttt{modalB} is not semisupertrue (use the same instance and transform as in the
proof for \texttt{spec}), while \texttt{modalD} and \texttt{modal4} are, so it is independent in
the variant of $\texttt{TMM} \setminus \{\texttt{ALLeq}\}$ using DB4 instead of T5.
On the other hand, \texttt{modalD} is provable from
$\texttt{modK} \cup \{ \texttt{EQrefl}, \texttt{denot} \}$.
Finally, in $\texttt{modK} \cup \{ \texttt{modalB}\}$, the axiom schemes \texttt{modal4} and
\texttt{modal5} (which are supertrue) are equivalent, so in any system containing
$\texttt{modK} \cup \{ \texttt{modalB} \}$, they are either both provable or both independent.

\subsection{The vacuous generalization axiom scheme}

$\texttt{vacGen}: \phi \to \forall x \phi \comma \DV(x, \phi)$
\vspace{0.5em}

The scheme \texttt{vacGen} says that one can universally quantify a formula over a variable not
occurring in it.
Over classical calculus, it is equivalent to its dual $\exists x \phi \to \phi \comma \DV(x, \phi)$
and to $\exists x \phi \to \forall x \phi \comma \DV(x, \phi)$ (with the help of \texttt{spec} for
the reverse implication).

\subsection{The equality bloc}

$\texttt{EQrefl}: x \equiv x$\\
$\texttt{EQsymm}: x \equiv y \to y \equiv x$\\
$\texttt{EQtrans}: x \equiv y \to (y \equiv z \to x \equiv z)$
\vspace{0.5em}

In \texttt{TMM}, we used the characterization of an equivalence relation as being a reflexive, transitive, symmetric one.
One can also use the characterization: reflexive and (left or right)-Euclidean: both characterizations are equivalent over \texttt{minimplcalc}.
Recall that a binary relation~$\equiv$ is left (resp.\ right)-Euclidean if it satisfies~\texttt{ax-$\equiv_2$} (resp.~\texttt{ax-$\equiv_1$} or \texttt{EQeucl}) of Table~\ref{tab:labels} of Appendix~\ref{app:labels}.
As their names suggest, these axiom schemes can also be seen as the predicate axiom schemes associated with the predicate $\equiv$.

The independence proofs are similar, since this bloc is fairly well separated from the rest of the axiom schemes: when \texttt{EQeucl} replaces $\{ \texttt{EQsymm}, \texttt{EQtrans} \}$ in the axiomatization, the model given above for the independence of \texttt{EQrefl} (resp.\ \texttt{EQtrans}) proves the independence of \texttt{EQrefl} (resp.\ \texttt{EQeucl}).

In the presence of right-Euclideanness, one can weaken the requirement of reflexivity to that of right-seriality, that is, the property expressed by the axiom $\texttt{denot}'$.
See the next subsection for that axiomatization.

One can weaken \texttt{EQsymm} by adding the DV~condition $\DV(x, y)$, and recover the full scheme from it and \texttt{monadic}.
One can also weaken \texttt{EQtrans} by adding the DV~conditions $\DV(x, y)$ and $\DV(x, z)$ (or $\DV(x, z)$ and $\DV(y, z)$), and recover the full scheme from it and \texttt{monadic}.
We do not know if we can fully unbundle \texttt{EQtrans}, that is, require that all variables be disjoint.

We can also take the universal closures of these axiom schemes, since \texttt{spec} and \texttt{gen} show that they are equivalent to them.

\subsection{The denotation axiom scheme}

$\texttt{denot}:
x \equiv x \to \neg\forall y \neg y \equiv x \comma \DV(x, y)$
\vspace{0.5em}

Tarski used the axiom scheme $\texttt{denot}'$: $\exists x x \equiv y \comma \DV(x, y)$ in place\footnote{Tarski
used it both with and without the DV~condition.
For a logic with terms, the corresponding axiom scheme should be $\exists x x \equiv t \comma \DV(x, t)$ (see~\cite[\S~4]{km}) and there, the DV~condition is necessary (else one could apply the substitution $t \leftarrow \{x\}$ in a well-founded set theory).
This gives a motivation for keeping the DV~condition in the term-less case, and proving that it is sufficient.}
of $\texttt{denot}$.
More precisely, Tarski's and the present axiomatizations correspond respectively to the left and right hand side of the bi-entailment $\{\texttt{EQeucl}, \texttt{denot}'\} \dashv\vdash_\texttt{modK} \{\texttt{EQrefl}, \texttt{EQsymm},  \texttt{EQtrans}, \texttt{denot}\}$.

We chose our axiom schemes in order to better separate the equality bloc from the rest of the axiom schemes.
In \texttt{denot}, the addition of the antecedent $x \equiv x$ could be seen as a cheap trick to guarantee independence of \texttt{EQrefl}, but more interestingly, \texttt{denot} can be seen as existential generalization applied to the unary predicate $\bm{\cdot} \equiv x$.
It is also related to the proposed view in free logic that an object exists (equivalently, denotes) if and only if it is equal to itself.

As noted above, $\texttt{modK} \cup \{ \texttt{EQrefl}, \texttt{denot} \}$ proves \texttt{modalD},
so if one wanted to use for the modal bloc the axiomatization DB4 and keep independence
of~\texttt{modalD}, then one could weaken \texttt{denot} to its generalization
$\forall x (x = x \to \exists y y = x) \comma \DV(x, y)$.

Finally, one could take as axiom the contrapositive $\forall y \neg y \equiv x \to \neg x \equiv x$,
which can be seen as universal instantiation for the predicate ``$\neg \:\bm{\cdot} \equiv x$\rlap{.}''
Then, independence of \texttt{contrap} could be simply proved by considering the
imp$_\equiv$-valuation which ignores quantifiers and is always false on negations.

\subsection{The substitution axiom scheme}

$\texttt{subst}:
\forall x (x \equiv y \to (\phi \to \forall x ( x \equiv y \to \phi))) \comma \DV(x, y)\comma \DV(y, \phi)$
\vspace{0.5em}

One can replace the axiom schemes \texttt{spec} and \texttt{subst} by a version of \texttt{subst}
without its initial quantifier, call it
\href{https://us.metamath.org/mpeuni/ax12v2.html}{\texttt{ax12v2}} as in \texttt{set.mm}, which
also contains the closely related \href{https://us.metamath.org/mpeuni/ax12v.html}{\texttt{ax12v}}
and \href{https://us.metamath.org/mpeuni/ax-12.html}{\texttt{ax-12}}.
One has $\{\texttt{spec}, \texttt{subst}\} \dashv\vdash_\texttt{T} \texttt{ax12v} \dashv\vdash_\texttt{T} \texttt{ax12v2}\dashv_\texttt{T} \texttt{ax-12}$ (with converse entailment when \texttt{genEq}
or \texttt{ax-13} is added).
The schemes \texttt{ax12v}, \texttt{ax12v2}, \texttt{ax-12} are independent in the new systems, by
the same proof of Monk.
Independence of \texttt{gen} is also much easier to show in these systems: use the pc-valuation such
that $\val(\forall x \Phi)=0$ if and only if $x \in \OC(\Phi)$, as in~\cite{km}.

The scheme \texttt{ax12v} (as well as \texttt{ax12v2} and \texttt{ax-12}) is not supertrue, as can
be seen directly or because \texttt{spec}, which is not supertrue, is implied by it over the
supertrue set~\texttt{T}.

The scheme \texttt{subst} is still true if we remove the DV condition $\DV(x, y)$, and its degenerate
instance $\forall x (x \equiv x \to (\phi \to \forall x (x \equiv x \to \phi)), \DV(x, \phi)$
is equivalent over $\{\texttt{minimplcalc}, \texttt{gen}, \texttt{spec}, \texttt{EQrefl} \}$
to \texttt{vacGen}.

\subsection{The ``quantification over equal variables'' axiom scheme}

$\texttt{ALLeq}:
\forall x x \equiv y \to (\forall x \phi \to \forall y \phi)$
\vspace{0.5em}

This axiom (\href{https://us.metamath.org/mpeuni/ax-c11.html}{\texttt{set.mm/ax-c11}}) was
suggested to me by Norman Megill to add it to TMM, which otherwise is probably not complete.
In \texttt{set.mm}, it is implied by
\href{https://us.metamath.org/mpeuni/ax-12.html}{\texttt{set.mm/ax-12}}, which was recognized by
several to be ``too strong'' since it conveys the content of \texttt{ALLeq} in addition to the
substitution property conveyed by \texttt{subst}.

This axiom says that one can indifferently quantify over any of two ``always equal variables\rlap{,}''
and this phrase is in turn justified by the statement
\href{https://us.metamath.org/mpeuni/ax-c11n.html}{\texttt{set.mm/ax-c11n}}:
$\forall x x \equiv y \to \forall y y \equiv x$ showing the symmetry of that relation and
provable from $\texttt{modK} \cup \{\texttt{EQsymm}, \texttt{ALLeq}\}$ (while its reflexivity is
provable from $\{\texttt{ax-gen}, \texttt{EQrefl}\}$ and its transitivity is
provable from $\texttt{modK} \cup \{\texttt{EQtrans}, \texttt{ALLeq}\}$).

\subsection{The generalized equality axiom scheme}

$\texttt{genEq}:
\neg \forall x x \equiv y \to (\neg \forall x x \equiv z \to (y \equiv z \to \forall x y \equiv z))$
\vspace{0.5em}

Informally, \texttt{genEq} says that one can universally quantify an equality over a variable which
does not occur in that equality.
However, it is not a consequence of \texttt{vacGen}, because this non-occurrence is expressed by
the antecedents (on a domain with at least two elements, the formula $\forall v_0 v_0 \equiv v_1$
is false), and not by a DV~condition as in \texttt{vacGen}.

The axiom scheme \texttt{genEq} used to be an axiom scheme of \texttt{set.mm} but it got replaced by
\href{https://us.metamath.org/mpeuni/ax12v2.html}{\texttt{ax-13}} (see Appendix~\ref{app:labels}).
One has
$\{\texttt{spec}, \texttt{genEq}\} \vdash_\texttt{T} \texttt{ax-13} \vdash_\texttt{T} \texttt{genEq}$.
Furthermore, the independence proofs given above still work for that system.

The scheme \texttt{ax-13} is not semisupertrue (make the $\{x, y\}$-transform or the
$\{x, z\}$-transform), so in the \texttt{set.mm}-variant of \texttt{TMM}, our results prove the
independence of \texttt{ax-13} from \{\texttt{mp}, \texttt{gen}, \texttt{ax-1}, \dots, \texttt{ax-11}\}.

\subsection{The predicate axiom schemes}

$\texttt{ax-${P_i}_j$}:
x \equiv y \to (P_i(z_1, \ldots, x, \ldots, z_{a_i}) \to P_i(z_1, \ldots, y, \ldots, z_{a_i}))$
\vspace{0.5em}

Theses axiom schemes ensure the substitutivity property of the predicate $P_i$ with respect to each of
its $a_i$ variables.
One can add to each predicate axiom scheme the DV~condition $\DV(x, y)$ on the first two variables.
One can also combine the $a_i$ predicate axiom schemes associated with an $a_i$-ary predicate,
obtaining for example in the case of the binary infix operator~$\in$ the axiom
$x_0 \equiv x_1 \to (y_0 \equiv y_1 \to (x_0 \in y_0 \to x_1 \in y_1))$.

\section{Mario Carneiro's proof of the independence of the rule of generalization}
\label{app:gen}

Here is a proof of the independence (but not object-independence) of the rule of generalization due
to Mario Carneiro.
It is simpler than the proof in the main part and gives additional insight to the concept of supertruth.

Fix a model of first-order logic and let $a \in D$ be a fixed element of the domain of discourse.
Let $P$ be a unary predicate and interpret it as being true at~$a$ and false at at least one element
of~$D$\rlap{.}\footnote{Equivalently, one could add a term~$t$, to be assigned to~$a$, and define the
predicate $P$ to be ``$\bm{\cdot} \equiv t$\rlap{.}''}
Define a formula as being $*$-true if it is true in that model as soon as~$v_0$\footnote{
or: as soon as at least one variable} is assigned to~$a$\rlap{.}\footnote{Therefore, substitution of
(individual) variables does not preserve $*$-truth, but this is not a problem: only substitution of
metavariables has to preserve $*$-truth, which is the case here by construction.}
Since $*$-truth is a weaker notion than truth, all hypothesis-free axioms are $*$-true.
It is not hard to verify that modus ponens~\texttt{mp} preserves $*$-truth.
However, the rule of generalization~\texttt{gen} does not preserve $*$-truth.
For instance, the formula $P(v_0)$ is $*$-true (since $v_0$ has to be assigned to~$a$, where $P$ is
true), but the formula $\forall v_0 P(v_0)$ is not (here, no variable need be assigned, since this is
a closed formula).

Equivalently, one can rephrase the proof in terms of formula transformation: say that the
formula~$\Phi$ is $*$-true if the formula $P(v_0) \to \Phi$ is true in~$D$.

This proof illustrates the fact that, although it is standard to regard a formula as true when it is a
true statement for all assignments of its free variables, there are many other possible interpretations.

\section{Table of correspondence of scheme labels}
\label{app:labels}

We gather in Table~\ref{tab:labels} some schemes used in this article and in the Metamath database
\href{https://us.metamath.org/mpeuni/mmset.html}{\texttt{set.mm}} with their labels.

\begin{table}[!ht]
\begin{center}
\caption{Correspondence between some scheme labels}
\label{tab:labels}
\begin{tabular}{|c|c|c|}
\hline
scheme & this article & \href{https://us.metamath.org/mpeuni/mmset.html}{\texttt{set.mm}}\\
\hline
$\phi \andd \phi \to \psi \rimplies \psi$ & \texttt{mp} & \texttt{ax-mp}\\
$\phi \to ((\psi\to\chi) \to (((\theta\to\psi) \to (\chi\to\tau)) \to (\psi\to\tau)))$ & \texttt{minimp} &
\texttt{minimp}\\
$(\phi\to\psi) \to ((\psi\to\chi) \to (\phi\to\chi))$ & \texttt{syl}, $\texttt{B}'$ & \texttt{imim1}\\
$(\phi\to\psi) \to ((\chi\to\phi) \to (\chi\to\psi))$ & $\texttt{syl}^*$, \texttt{B} & \texttt{imim2}\\
$(\phi \to (\psi \to \chi)) \to (\psi \to (\phi \to \chi))$ & \texttt{comm}, \texttt{C} & \texttt{pm2.04}\\
$\phi \to (\psi \to \phi)$ & \texttt{simp}, \texttt{K} & \texttt{ax-1}\\
$\phi \to \phi$ & \texttt{id}, \texttt{I} & \texttt{id}\\
$(\phi \to (\phi \to \psi)) \to (\phi \to \psi))$ & \texttt{hilbert}, \texttt{W} & \texttt{pm2.43}\\
$(\phi \to (\psi \to \chi)) \to ((\phi \to \psi) \to (\phi \to \chi))$ & \texttt{frege}, \texttt{S} &
\texttt{ax-2}\\
$((\phi \to \psi) \to \phi) \to \phi$ & \texttt{peirce}, \texttt{P} & \texttt{peirce}\\
\hline
$(\phi \to \neg \psi) \to (\psi \to \neg \phi)$ & \texttt{contrap} & \texttt{con2}\\
$(\neg\phi \to \neg\psi) \to (\psi \to \phi)$ & & \texttt{ax-3}\\
$\neg\phi \to (\phi \to \psi)$ & \texttt{notelim} & \texttt{pm2.21}\\
$\phi \to (\neg\phi \to \psi)$ & \texttt{excontra} & \texttt{pm2.24}\\
$(\neg\phi \to \phi) \to \phi$ & \texttt{clavius} & \texttt{pm2.18}\\
$\phi \to \neg\neg\phi$ & \texttt{notnotintro} & \texttt{notnot}\\
$\neg\neg \phi \to \phi$ & \texttt{notnotelim} & \texttt{notnotr}\\
\hline
$\phi \rimplies \forall x \phi$ & \texttt{gen} & \texttt{ax-gen}\\
$\forall x (\phi \to \psi) \to (\forall x \phi \to \forall x \psi)$ & \texttt{ALLdistr}, \texttt{modalK} &
\texttt{ax-4}\\
$\forall x \phi \to \phi$ & \texttt{spec}, \texttt{modalT} & \texttt{sp}\\
$\neg \forall x \phi \to \forall x \neg \forall x \phi$ & \texttt{modal5} & \texttt{ax-10}\\
$\forall x \neg \phi \to \neg \forall x \phi$ & \texttt{modalD} & \texttt{bj-modald}\\
$\neg \phi \to \forall x \neg \forall x \phi$ & \texttt{modalB} & \texttt{bj-modalb}\\
$\forall x \phi \to \forall x \forall x \phi$ & \texttt{modal4} & \texttt{hba1}\\
\hline
$\forall x \forall y \phi \to \forall y \forall x \phi$ & \texttt{ALLcomm} & \texttt{ax-11}\\
\hline
$\phi \to \forall x \phi \comma \DV(x, \phi)$ & \texttt{vacGen} & \texttt{ax-5}\\
$\exists x \phi \to \forall x \phi \comma \DV(x, \phi)$ & & \texttt{ax5ea}\\
\hline
$x \equiv x$ & \texttt{EQrefl} & \texttt{equid}\\
$x \equiv y \to y \equiv x$ & \texttt{EQsymm} & \texttt{equcomi}\\
$x \equiv y \to (y \equiv z \to x \equiv z)$ & \texttt{EQtrans} & \texttt{equtr}\\
$x \equiv y \to (x \equiv z \to y \equiv z)$ & \texttt{EQeucl}, \texttt{ax-$\equiv_1$} & \texttt{ax-7}\\
$x \equiv y \to (z \equiv x \to z \equiv y)$ & \texttt{ax-$\equiv_2$} & \texttt{equeucl}\\
\hline
$\neg \forall x \neg x \equiv y$ & & \texttt{ax-6}\\
$\neg \forall x \neg x \equiv y \comma \DV(x, y)$ & $\texttt{denot}'$ & \texttt{ax6v}\\
$x \equiv x \to \neg \forall y \neg y \equiv x \comma \DV(x, y)$ & \texttt{denot}\phantom{'} &
\texttt{bj-denot}\\
\hline
$x \equiv y \to (\forall y \phi \to \forall x (x \equiv y \to \phi))$ & & \texttt{ax-12}\\
$\forall x (x \equiv y \to (\phi \to \forall x (x \equiv y \to \phi)))\comma \DV(x, y)\comma \DV(y, \phi)$
& \texttt{subst} & \texttt{bj-ax12v}\\
\hline
$\forall x x \equiv y \to \forall y y \equiv x$ &  & \texttt{ ax-c11n}\\
$\forall x x \equiv y \to (\forall x \phi \to \forall y \phi)$ & \texttt{ALLeq} & \texttt{ax-c11}\\
$\forall x x \equiv y \to (\phi \to \forall x \phi) \comma \DV(x, y)$ & \texttt{oneObj} & \texttt{ax-c16}\\
\hline
$\neg x \equiv y \to (y \equiv z \to \forall x y \equiv z)$ & & \texttt{ax-13}\\
$\neg \forall x x \equiv y \to (\neg \forall x x \equiv z \to (y \equiv z \to \forall x y \equiv z))$ &
\texttt{genEq}, $\texttt{gen}_\equiv$ & \texttt{ax-c9}\\
\footnotesize{$\neg \forall x x \equiv x_1 \to ( \dots \to (\neg \forall x x \equiv x_n \to (P(x_1, \dots, x_n) \to \forall x P(x_1, \dots, x_n))) \dots)$}
& \texttt{gen$_P$} & $\sim$\texttt{ ax-c14}\phantom{$\sim$}\\
\hline
$x \equiv y \to (x \in z \to y \in z)$ & \texttt{ax-$\in_1$} & \texttt{ax-8 }\\
$x \equiv y \to (z \in x \to z \in y)$ & \texttt{ax-$\in_2$} & \texttt{ax-9 }\\
$x \equiv y \to (P_i(z_1, \ldots, x, \ldots, z_{a_i}) \to P_i(z_1, \ldots, y, \ldots, z_{a_i}))$ &
\texttt{ax-${P_i}_j$} &\\
\hline
\end{tabular}
\end{center}
\end{table}

\printbibliography

\end{document}